\documentclass{amsart}

\usepackage[T1]{fontenc}
\usepackage{mathtools}
\usepackage{subfigure}
\usepackage{graphicx}
\usepackage{caption}
\usepackage{color,soul,amsmath,amssymb,amsthm,amscd}
\usepackage{comment}

\usepackage{xcolor}

\usepackage{amsmath}
\usepackage{amssymb}
\newtheorem{theorem}[subsection]{Theorem}

\newtheorem{corollary}[subsection]{Corollary}

\newtheorem{proposition}[subsection]{Proposition}

\title{Direction-Critical Configurations in Noncentral General Position}
\author{Silvia Fern\'andez-Merchant and Rimma H\"am\"al\"ainen}

\begin{document}
\maketitle

\begin{abstract}
In 1982, Ungar proved that the connecting lines of a set of $n$ noncollinear points in the plane determine at least $2\lfloor n/2 \rfloor$ directions (slopes). Sets achieving this minimum for $n$ odd (even) are called \emph{direction-(near)-critical} and their full classification is still open. To date, there are four known infinite families and over 100 sporadic critical configurations. 
Jamison conjectured that any direction-critical configuration with at least 50 points belongs to those four infinite families. Interestingly, except for a handful of sporadic configurations, all these configurations are centrally symmetric.  We prove Jamison's conjecture, and extend it to the  near-critical case, for centrally symmetric configurations in \emph{noncentral general position}, where only the connecting lines through the center of symmetry may pass through more than two points. As in Ungar's proof, our results are proved in the more general setting of \emph{allowable sequences}. We show that, up to equivalence, the \emph{central signature} of a set uniquely determines a centrally symmetric direction-(near)-critical allowable sequence in noncentral general position, and classify such allowable sequences that are geometrically realizable.
\end{abstract}

\section{Introduction}
In 1970, Scott \cite{S70} proposed the problem of finding the least number of directions (slopes) determined by a set of $n$ points in the plane, not all collinear. He conjectured that the minimum number of different slopes determined by the connecting lines of a set of $n$ noncollinear points in the plane is $2\lfloor n/2 \rfloor$. This conjecture was completely settled by Ungar  in 1982 \cite{UNG}. Sets with $n$ points and $n-1$ directions are called \emph{direction-critical} (or  \emph{slope-critical}); and sets with $n$ points and $n$ directions are called \emph{direction-near-critical}. Hence, the sets achieving the minimum in Ungar's Theorem are the direction-critical configurations, which always have an odd number of points and so we call them \emph{odd-critical}; and the direction near-critical configurations with an even number of points, which we call \emph{even-near-critical}. 

After Ungar's result, the focus naturally switched to the classification of direction (near)-critical configurations. In 1983, Jamison and Hill \cite{JAMH1983} described four infinite families and over a hundred sporadic odd-critical configurations (details in Section \ref{sec:classification}). They noted that the four infinite families and all but four sporadic configurations are centrally symmetric. Jamison conjectured that any odd-critical configuration with at least 50 points belongs to those four families \cite{JAM1984}. He also proved that any near-critical set of points in general position is an affine copy of the vertices of a regular polygon\cite{JAM1984,JAM1986}. In his analysis, Jamison defines a \emph{dividing line} of a set of points $P$ as a line connecting points of $P$ that divides the rest of $P$ in half. These lines are best known as \emph{halving lines} of $P$. The cyclic sequence of positive integers corresponding to half the number of points  per halving line of $P$ minus its center, read in counterclockwise order, is the \emph{central signature} of $P$.

Inspired by this work, we study centrally symmetric configurations. We extend the notion of general position to sets in \emph{non-central general position}, where the only lines possibly passing through more than two points of the set are its halving lines. We completely classify the centrally symmetric configurations in noncentral general position that achieve the minimum in Ungar's Theorem. More precisely, in Theorem \ref{th:noncentral_2}, we prove that all but one centrally symmetric odd-critical configuration of points in noncentral general position belongs to one of the 4 families described in Jamison's Conjecture. All even-near-critical configurations are obtained from the odd-critical by removing their center of symmetry. 

Moreover, as in Ungar's Theorem, our analysis extends to the setting of \emph{allowable sequences} \cite{GP80,GP81} (definition and more details is Section \ref{sec:allow_seq}). Goodman and Pollack \cite{GP84} showed that any allowable sequence corresponds to a \emph{generalized configuration of points}, that is, a set of points together with a pseudoline arrangement so that every pair of points is contained in a unique pseudoline. Note that the set of configurations of points is then a proper subset of the set of generalized configurations of points/allowable sequences. 
In contrast to our geometric classification, we prove in Theorem \ref{th:main_pseudo} that, up to combinatorial equivalence, any cyclic sequence of positive integers is the central signature of a unique even-near-critical and (by adding the center of symmetry) a unique odd-critical centrally symmetric allowable sequence in noncentral general position. Only a few of these allowable sequences are geometrically realizable, as shown by Theorem \ref{th:noncentral_2}.

The paper is organized as follows. In Section \ref{sec:allow_seq}, we present the definition of allowable sequence, related notation and results. In Section \ref{sec:dir_critical_allowable_seq}, 
we determine the structure of any even-near-critical centrally symmetric allowable sequence in noncentral general position in terms of its central signature, Theorem \ref{th:uniqueness} and Corollary \ref{th:uniqueness_general}; and show their existence in Theorem \ref{th:existence}. In Section \ref{sec:classification}, we describe some known results about the classification of odd-critical and even-near critical geometric configurations and state Theorem \ref{th:noncentral_2}, which fully determines the even-near critical, and thus the odd-critical, allowable sequences from Theorem \ref{th:main_pseudo} that are geometrically realizable. The proof of this result follows from Theorems \ref{th:(>2,>1,...)}-\ref{th:(2,1,2)}, that we present in Section \ref{sec:proofs}.

\section{Allowable sequences}\label{sec:allow_seq}
Let $P$ be a set of $n$ points in the plane. The \emph{circular sequence} $\Pi(P)$ associated with $P$ is a doubly-infinite sequence of permutations of the points in $P$  determined by the projections of $P$ onto a line that rotates around a circle enclosing $P$. As an abtract generalization of a circular sequence, an \emph{allowable sequence} of $n$ points is a doubly-infinite periodic sequence of permutations $\Pi=\{\pi_i\}_{i=-\infty}^{\infty}$ of the set of \emph{points} $[n]:=\{1,2,3,\ldots,n\}$, satisfying the following properties:
\begin{enumerate}
    \item By relabeling the points, it can be assumed that $\pi_0=(1,2,3,\ldots, n)$.
    \item There is $h\in \mathbf{Z}^+$ such that $\pi_{i+h}$ is the reversal of $\pi_i$ for every $i\in\mathbf{Z}$. Then $\Pi$ has period $2h$ and $\{\pi_0,\pi_1,\ldots,\pi_h\}$ is a \emph{halfperiod} of $\Pi$ of \emph{length} $h=h(\Pi)$.
    \item $\pi_{i}$ is obtained from $\pi_{i-1}$ by the reversal of one or more disjoint substrings, each involving consecutive elements of $\pi_{i-1}$. These reversals are called \emph{switches}, and the set of switches occuring from $\pi_{i-1}$ to $\pi_i$ is the \emph{$i^{th}$ move} of $\Pi$.
    \item Any pair of points participates in a switch exactly once within a halfperiod.
\end{enumerate}

If $\Pi$ is an allowable sequence on the set of points $[n]$ and $\mathcal{S}\subseteq [n]$, then the \emph{allowable sequence induced by $\mathcal{S}$}, denoted by $\Pi|_{\mathcal{S}}$, is obtained from $\Pi$ by deleting the points not in $\mathcal{S}$ from each permutation and removing repeated permutations.

Allowable sequences were used by Goodman and Pollack to approach combinatorial problems of sets of points in the 1980s \cite{GP80,GP81}. Since circular sequences are especial cases of allowable sequeces, every set of points corresponds to an allowable sequence but not every allowable sequence is the circular sequence of a set of points. In fact, Goodman and Pollack \cite{GP84} showed that up to combinatorial equivalence, there is a one-to-one correspondence between the set of allowable sequences and the set of generalized configurations of points.
In this new setting, all switches occurring between consecutive permutations of an allowable sequence correspond to \emph{pseudolines determining the same direction}. That is, the number of directions determined by an allowable sequence $\Pi$ is the length $h(\Pi)$ of its halfperiod. Ungar proved that if $\Pi$ is an allowable sequence  with $n$ points, then $h(\Pi)\geq 2\lfloor n/2\rfloor$. In other words, the odd-critical and even-near-critical allowable sequences are those with $n$ points and half-period of length $2\lfloor n/2\rfloor$.

Since odd centrally symmetric configurations are precisely those obtained from even ones by adding their center of symmetry, we only consider even configurations. All relevant concepts are naturally extended to allowable sequences. Let $\Pi$ be an allowable sequence of $2n$ points. For a permutation $\pi$ and a point $p$ of $\Pi$, the position of $p$ in $\pi$ is denoted by $\pi(p)$. $\Pi$ is  \emph{centrally symmetric} if for every point $p$ there is a point $\overline{p}$ such that   $\pi(p)+\pi(\overline{p})=2n+1$  for any permutation $\pi\in\Pi$. The points $p$ and  $\overline{p}$ are centrally symmetric and they are said to be \emph{conjugates}. So  $\overline{p}$ is the conjugate of $p$, and $p=\overline{\overline{p}}$ is the conjugate of  $\overline{p}$. The switches reversing a centered substring of a permutation are called \emph{crossing switches}. They correspond to the halving lines of a set of points and they all pass through the center of symmetry of the set.
The sequence $(d_1,d_2,\dots,d_t)$, where $2d_i$ is the number of points reversed by the $i^{th}$ crossing switch in the halfperiod $\pi_0,\pi_1\dots,\pi_h$ of $\Pi$, is the \emph{central signature} of $\Pi$ and $t$ is its \emph{central degree}. Allowable sequences all whose switches are \emph{transpositions} (switches of two points) are said to be in \emph{general position}. Allowable sequences in which all switches, except perhaps for the crossing switches, are transpositions are said to be in \emph{noncentral general position}. (The central signature of an allowable sequence is also called a \emph{crossing distance partition} \cite{FMH}).

We consider the following questions: Which cyclic positive integer sequences are the central signature of an odd-critical or even-near-critical centrally symmetric allowable sequence in noncentral general position? Which of such allowable sequences are \emph{geometrically realizable}, that is, they are the circular sequence of a set of points? 

In the rest of the paper,  $\Pi=\{\pi_i\}_{i\in\mathbb{Z}}$ is an even-near-critical centrally symmetric allowable sequence with $2n$ points in noncentral general position and with central signature $(d_1,d_2,\dots,d_t)$ for some $t\geq 2$. We assume that $\pi_0=(1,2,\dots,n-1,n,\,\overline{n},\overline{n-1},\dots,\overline{2},\overline{1} )$ and that the first crossing switch occurs in the first move (from $\pi_0$ to $\pi_1$).  The $i^{th}$ crossing switch in the halfperiod starting at $\pi_0$ reveres a centered substring $s_i$ of $2d_i$ points, which we call a \emph{crossing substring}. Before reversing, 
\begin{align*}
    s_i&=(s_i(1),s_i(2),\dots, s_i(d_i),s_i(d_i+1),s_i(d_i+2),\dots, s_i(2d_i))\\
    &=(s_i(1),s_i(2),\dots, s_i(d_i),\overline{s_i(d_i)},\overline{s_i(d_i-1)},\dots, \overline{s_i(2d_i)}).
\end{align*}

Figure \ref{fig:321} shows an example of a halfperiod of an even-near-critical centrally symmetric allowable sequence in noncentral general position with $2n=12$ points and central signature $(3,2,1)$. Its reversed crossing substrings are highlighted to easily identify the crossing switches. This sequence is not geometrically realizable. 
\setcounter{MaxMatrixCols}{20}
\begin{figure}[h]
    \centering
\begin{footnotesize}
\begin{equation*}
    \{\pi_i\}_{i=0}^{12}=\begin{pmatrix}
    1 & 2 & 3 & 4 & 5 & 6 & \overline{6} & \overline{5} & \overline{4} & \overline{3} & \overline{2} & \overline{1}\\
1 & 3 & 2 & \colorbox{green!35}{$\overline{4}$} & \colorbox{green!35}{$\overline{5}$} & \colorbox{green!35}{$\overline{6}$} & \colorbox{green!35}{$6$} & \colorbox{green!35}{$5$} & \colorbox{green!35}{$4$} & \overline{2} & \overline{3} & \overline{1}\\
    3 & 1 & \overline{4} & 2 & \overline{5} & \overline{6} & 6 & 5 & \overline{2} & 4 & \overline{1} &\overline{3}\\
    3 & \overline{4} & 1 & \overline{5} & 2 & \overline{6} & 6 & \overline{2} & 5 & \overline{1} & 4 & \overline{3}\\
    \overline{4} & 3 & \overline{5} & 1 & \overline{6} & 2 & \overline{2} & 6 & \overline{1} & 5 & \overline{3} & 4 \\
    \overline{4} & \overline{5} & 3 & \overline{6} & 1 & 2 & \overline{2} & \overline{1} & 6 & \overline{3} & 5 & 4\\
    \overline{4} & \overline{5} & \overline{6} & 3 & \colorbox{green!35}{$\overline{1}$} & \colorbox{green!35}{$\overline{2}$} & \colorbox{green!35}{$2$} & \colorbox{green!35}{$1$} & \overline{3} & 6 & 5 & 4 \\
    \overline{4} & \overline{5} & \overline{6} & \overline{1} & 3 & \overline{2} & 2 & \overline{3} & 1 & 6 & 5 & 4 \\
    \overline{4} & \overline{5} & \overline{1} & \overline{6} & \overline{2} & 3 & \overline{3} & 2 & 6 & 1 & 5 & 4 \\
    \overline{4} & \overline{1} & \overline{5} & \overline{2} & \overline{6} & \colorbox{green!35}{$\overline{3}$} & \colorbox{green!35}{$3$} & 6 & 2 & 5 & 1 & 4\\
    \overline{1} & \overline{4} & \overline{2} & \overline{5} & \overline{3} & \overline{6} & 6 & 3 & 5 & 2 & 4 & 1 \\
    \overline{1} & \overline{2} & \overline{4} & \overline{3} & \overline{5} & \overline{6} & 6 & 5 & 3 & 4 & 2 & 1\\
    \overline{1} & \overline{2} & \overline{3} & \overline{4} & \overline{5} & \overline{6} & 6 & 5 & 4 & 3 & 2 & 1
    \end{pmatrix}
\end{equation*}
\end{footnotesize}
\caption{An even-near-critical centrally symmetric allowable  in noncentral general position and sequence with central signature $(3,2,1)$.}
    \label{fig:321}
\end{figure}
\section{Direction-critical allowable sequences}\label{sec:dir_critical_allowable_seq}
As a consequence of Ungar's proof \cite{UNG}, a centrally symmetric allowable sequence $\Pi$ is even-near-critical if and only if between the $i^{th}$ and $(i+1)^{th}$ crossing switches, there are exactly $d_i+d_{i+1}-1$ permutations. In this case, $n=d_1+d_2+\dots+d_t$. 

In order to understand the structure of the allowable sequences at hand, we define the \emph{path of the point} $p$ in $\Pi$, denoted by $\gamma(p)$, as a sequence of letters $\textsc{c}$, $\textsc{p}$, $\textsc{r}$ and $\textsc{l}$ of length $2n$ that tracks the movement of $p$ through the halfperiod of $\Pi$ starting at $\pi_0$. More precisely, the $i^{th}$ entry of $\gamma (p)$ is 
\begin{itemize}
    \item $\textsc{c}$ if $p$ participates in a crossing switch from $\pi_{i-1}$ to $\pi_i$, this is a \emph{central jump};
    \item $\textsc{p}$ if $p$ does not participate in any switch from $\pi_{i-1}$ to $\pi_i$, this is a \emph{passive jump};
    \item $\textsc{r}$ if the position of $p$ in $\pi_i$ is to the right of that in $\pi_{i-1}$, this is a \emph{right jump};
    \item $\textsc{l}$ if the position of $p$ in $\pi_i$ is to the left of that $\pi_{i-1}$, this is a \emph{left jump}.
\end{itemize}
Similar notation was used by Jamison \cite{JAM1986}. Now we are ready to state our first result.

\begin{theorem}\label{th:uniqueness}
Let $\Pi$ be an even-near-critical centrally symmetric allowable sequence of $n$ points in noncentral general position equipped with the central signature $(d_1,d_2\dots,d_t)$. If the first move contains a crossing switch, then for $1\leq k \leq d_1$,
\begin{align*}
    \gamma(s_1(k))&=\textsc{c}\underbrace{\textsc{p}\cdots \textsc{p}}_{k-1}\underbrace{\textsc{r}\cdots \textsc{r}}_{n-d_1}\underbrace{\textsc{p}\cdots \textsc{p}}_{2(d_1-k)+1}\underbrace{\textsc{l}\cdots \textsc{l}}_{n-d_1}\underbrace{\textsc{p}\cdots \textsc{p}}_{k-1}, \text{ and so}\\
     \gamma(\overline{s_1(k)})&=\textsc{c}\underbrace{\textsc{p}\cdots \textsc{p}}_{k-1}\underbrace{\textsc{l}\cdots \textsc{l}}_{n-d_1}\underbrace{\textsc{p}\cdots \textsc{p}}_{2(d_1-k)+1}\underbrace{\textsc{r}\cdots \textsc{r}}_{n-d_1}\underbrace{\textsc{p}\cdots \textsc{p}}_{k-1}.\nonumber
\end{align*}
\end{theorem}

\begin{proof}
Consider the point $s_1(d_1)$ (see the Appendix for some examples). After participating in the first crossing switch, $s_1(d_1)$ can only change its position using transpositions (due to noncentral general position).  Also $s_1(d_1)$ cannot participate in a switch again until all other points in $s_1$ that are to the right of $s_1(d_1)$ in $\pi_1$ have moved. Since each of these points do not switch with each other in the rest of the halfperiod, they move one at a time and so $s_1(d_1)$ does not move in the next $d_1-1$ permutations. In other words, after the central jump, there must be $d_1-1$ passive jumps. Similarly, at the end of the halfperiod $s_1(d_1)$ moves back to the symmetric of its original position. Right after this happens, the rest of the points in $s_1$ must return to the symmetric of their original positions, which takes at least $d_1-1$ passive jumps for $s_1(d_1)$. Because there is only $2n$ permutations and $s_1(d_1)$ participates in $1$ crossing switch and $2n-2d_1$ transpositions, then $s_1(d_1)$ participates in exactly $2d_1-1$ passive jumps. That is, besides the $d_1-1$ passive jumps at the beginning and $d_1-1$ at the end of the halfperiod, there is only one more passive jump of $s_1(d_1)$. Note that a passive jump needs to take place before any change of direction of $s_1(d_1)$. This means that $s_1(d_1)$ only changes direction once and so all of its $n-d_1$ right jumps (one transposition with each of the points $\overline{1},\overline{2},\dots,\overline{n-d_1}$) are consecutive followed by a passive jump and then all of its $n-d_1$ left jumps (one transposition with each of the points $1,2,\dots , n-d_1$). Therefore,
    \begin{align*}
    \gamma(s_1(d_1))&=\textsc{c}\underbrace{\textsc{p}\cdots \textsc{p}}_{d_1-1}\underbrace{\textsc{r}\cdots \textsc{r}}_{n-d_1}\underbrace{\textsc{p}}_{1}\underbrace{\textsc{l}\cdots \textsc{l}}_{n-d_1}\underbrace{\textsc{p}\cdots \textsc{p}}_{d_1-1}.
\end{align*}
Note that this means that $\pi_{n+1}(s(d_1))=n+1+(n-d_1)=2n-d_1+1$. 
Since $s_1(d_1),s_1(d_1-1),\dots,s_1(2),s_1(1)$ remain in this order after the first crossing switch, then $s_1(1)$ is in position $2n$ in $\pi_{n+1}$. Then $s_1(k)$ participates in $d_1-k$ passive jumps right after getting to position $2n+1-k$ while waiting for $s_1(d_1-k+1),\dots,s_1(d_1-1),s_1(d_1)$ to get to boxes $2n+1-(k+1),\dots,2n+1-(d_1-1),2n+1-d_1$, respectively; one more passive jump waiting for $s_1(d_1)$ to change directions; and another $d_1-k$ passive jumps while waiting for $s_1(d_1),s_1(d_1-1),\dots,s_1(d_1-k+1)$ to get out of boxes $2n+1-d_1,2n+1-(d_1-1),\dots 2n+1-(k+1)$, respectively. Thus \begin{align*}
    \gamma(s_1(k))&=\textsc{c}\underbrace{\textsc{p}\cdots \textsc{p}}_{k-1}\underbrace{\textsc{r}\cdots \textsc{r}}_{n-d_1}\underbrace{\textsc{p}\cdots \textsc{p}}_{2(d_1-k)+1}\underbrace{\textsc{l}\cdots \textsc{l}}_{n-d_1}\underbrace{\textsc{p}\cdots \textsc{p}}_{k-1}.
\end{align*}
\end{proof}

By central symmetry and periodicity of $\Pi$, Theorem \ref{th:uniqueness} generalizes to paths of points in other crossing switches as follows.

\begin{corollary}\label{th:uniqueness_general}
Let $\Pi$ be an even-near-critical centrally symmetric centrally symmetric allowable sequence of $2n$ points in noncentral general position and with central signature $(d_1,d_2,\dots,d_t)$, and assume that the first move contains a crossing switch. Let $\delta_1=0$, $\delta_2=d_1+d_2$, and for $3\leq i\leq t$ let $\delta_i=d_1+2d_2+\dots+2d_{i-1}+d_i$ so that $s_i$ is reversed from $\pi_{\delta_i}$ to $\pi_{\delta_i+1}$. Then for $1\leq i\leq t$ and $1\leq k\leq d_i$, we have
\begin{equation*}
    \gamma(s_i(k))=
    \begin{cases} 
     \underbrace{\textsc{r}\dots \textsc{r}}_{\delta_i-k+1}\underbrace{\textsc{p}\cdots \textsc{p}}_{k-1}\textsc{c}\underbrace{\textsc{p}\cdots \textsc{p}}_{k-1}\underbrace{\textsc{r}\cdots \textsc{r}}_{n-d_i}\underbrace{\textsc{p}\cdots \textsc{p}}_{2(d_i-k)+1}\underbrace{\textsc{l}\dots \textsc{l}}_{\substack{n-d_i\\-\delta_i+k-1}}
     &\mbox{if } d_i-k<n-\delta_i,\\ 
     \underbrace{\textsc{p}\dots \textsc{p}}_{\substack{d_i-k\\+\delta_i-n+1}}\underbrace{\textsc{r}\dots \textsc{r}}_{n-d_i}\underbrace{\textsc{p}\cdots \textsc{p}}_{k-1}\textsc{c}\underbrace{\textsc{p}\cdots \textsc{p}}_{k-1}\underbrace{\textsc{r}\cdots \textsc{r}}_{n-d_i}\underbrace{\textsc{p}\cdots \textsc{p}}_{\substack{d_i-k\\-\delta_i+n}}
     & \mbox{if } d_i-k\geq |\delta_i-n|,\\
     \underbrace{\textsc{l}\dots \textsc{l}}_{\substack{\delta_i- n\\-d_i+k}}\underbrace{\textsc{p}\cdots \textsc{p}}_{2(d_i-k)+1}\underbrace{\textsc{r}\cdots \textsc{r}}_{n-d_i}\underbrace{\textsc{p}\cdots \textsc{p}}_{k-1}\textsc{c}\underbrace{\textsc{p}\cdots \textsc{p}}_{k-1}\underbrace{\textsc{r}\cdots \textsc{r}}_{2n-\delta_i-k}
     & \mbox{if }  d_i-k<\delta_i-n.
     \end{cases}
\end{equation*}
\end{corollary}

\begin{theorem}\label{th:existence}
Any positive integer sequence $(d_1,...,d_t)$ with $t>1$ is the central signature of an even-near-critical centrally symmetric allowable sequence $\Pi$ in non-central general position. 
\end{theorem}

\begin{proof}
We prove the result by induction on $d_1+d_2+\dots+d_t$. First note that if $d_i=1$ for all $1\leq i\leq t$, then $(d_1,...,d_t)$ is the central signature of the circular sequence of the vertices of a regular polygon, and this is the only possible partition when the sequence's sum is $2$. Let $n\geq 2$ and assume that there is a desired allowable sequence for any sequence with total sum at most $n$. Consider the sequence $d=(d_1,d_2,\dots,d_t)$ of positive integers with $d_1+d_2+\dots+d_t=n+1$. We argued that the result holds if all entries of $d$ are equal to $1$, so assume without loss of generality that $d_1\geq 2$. By induction, the sequence $d'=(d_1-1,d_2,d_3,\dots,d_t)$ is the central signature of an even-near-critical centrally symmetric allowable sequence $\Pi'=\{\pi_i\}_{i\in\mathbb{Z}}$ in noncentral general position. We construct the allowable sequence $\Pi$ from $\Pi'$ by adding a point $p'$ and its conjugate as follows.

Since $\Pi'$ is even-near-critical, then $\{\pi'_0,\pi'_1,\dots,\pi'_{2n}\}$ is a halfperiod. Assume that $\pi'_0=(1,2,\dots,n,\overline{n},\dots,\overline{2},\overline{1})$ and let $s'_1,s'_2,\dots, s'_t$ be the crossing substrings of $\Pi'$. Let $p=s'_1(1)=n-d_1+1$. Let $\pi[a,b]$ be the substring of $\pi$ consisting of all the elements in positions $a,a+1,\dots,b$. For $0\leq i \leq 2$, let $\pi_i[1,n+1]=$
\begin{equation*}
    \begin{cases}
     \pi'_i[1,p]*p'*\pi_i[p+1,n] & \text{ if } i=0,\\
     \pi'_i[1,p]*\overline{p'}*\pi_i[p+1,n] & \text{ if } i=1,2,\\
     \pi'_i[1,\pi_i(\overline{p})+1]*\overline{p'}*\pi_{i-1}[\pi_{i-1}(\overline{p})+1,n] & \text{ if } 3\leq i\leq p,\\
     \overline{p}*\overline{p'}*\pi'_{i-1}[2,n] & \text{ if } p+1\leq i \leq p+2d_1,\\
     \pi'_{i-2}[1,\pi'_{i-2}(\overline{p})+1]*\overline{p'}*\pi'_{i-1}[\pi'_{i-1}(\overline{p})+1,n] & \text{ if } p+2d_1+1\leq i \leq 2n+1,\\
     \pi'_{2n}[1,p]*\overline{p'}*\pi'_{2n}[p+1,n] & \text{ if } i=2(n+1)
    \end{cases}
\end{equation*}
where $*$ denotes the concatenation of strings. 
Finally, in order to guarantee central symmetry, let $\pi_i=\pi_i[1,n+1]*\overline{\pi_i[1,n+1}]$, where $\overline{\pi_i[1,n+1}]$ is the string whose elements are the conjugates of $\pi_i[1,n+1]$ in reversed order.

We first verify that $\Pi$ is an allowable sequence on $2(n+1)$ points with halfperiod $\{\pi_0,\pi_1,\dots,\pi_{2(n+1)}\}$. By Theorem \ref{th:uniqueness}, the paths of $p$ and $\overline{p}$ in $\Pi'$ are given by 
\begin{align*}
    & \gamma(p)=\gamma(s_1(1))= \textsc{c}\underbrace{\textsc{r}\dots \textsc{r}}_{n-d_1}\underbrace{\textsc{p}\dots \textsc{p}}_{2d_1-1}\underbrace{\textsc{l}\dots \textsc{l}}_{n-d_1}, \\
    & \gamma(\overline{p})=\gamma(\overline{s_1(1)})= \textsc{c}\underbrace{\textsc{l}\dots \textsc{l}}_{n-d_1}\underbrace{\textsc{p}\dots \textsc{p}}_{2d_1-1}\underbrace{\textsc{r}\dots \textsc{r}}_{n-d_1}.
\end{align*}
This means that $\pi'_i(\overline{p})=p-i+1$ for each $2\leq i \leq p$ and $\overline{p}$ transposes with some point $q_i$ to obtain $\pi'_i$ from $\pi'_{i-1}$; $\pi'_i(\overline{p})=1$ for each $p\leq i \leq p+2d_1-1$; and $\pi'_i(\overline{p})=i-p-2d_1+2$ for each $p+2d_1\leq i \leq 2n$ and $\overline{p}$ transposes with some point $q_i$ to obtain $\pi'_i$ from $\pi_{i-1}$.

First, $\pi_0[1,n+1]$ is obtained from $\pi'_0[1,n]$ by inserting the new point $p'$ right after the point $p$, so that $\pi_0$ is in fact a permutation of $2(n+1)$ points. Then $\pi_1$ is obtained from $\pi_0$ by the reversal of $s_1=(p,p',p+1,\dots,n,\overline{n},\dots,\overline{p+1},\overline{p'},\overline{p})$, which is a crossing switch of $n-p+2=d_1$ points. The permutation $\pi_2$ is obtained from $\pi_1$ by the same switches needed to obtain $\pi'_i$ from $\pi'_{i-1}$. In other words, $\pi_1[1,n+1]$ is obtained from $\pi'_1[1,n]$ by inserting the point $p'$ right after the point $q_1$.  For $3\leq i \leq p$, $\pi_i[1,n+1]$ is obtained from $\pi_{i-1}[1,n+1]$ by all the switches needed to obtain $\pi'_i[1,n]$ from $\pi'_{i-1}[1,n]$ that occur to the left of $p$, the transposition $(p,q_i)$ (which is also needed to go from $\pi'_{i-1}[1,n]$ to $\pi'_i[1,n]$), the transposition $(p',q_{i-1})$, and all the switches needed to obtained $\pi'_{i-1}[1,n]$  from $\pi'_{i-2}[1,n]$ that occur to the right of $p$. For $p+1\leq i\leq p+2d_1$, $\pi_i[1,n+1]$ is obtained from $\pi_{i-1}[1,n+1]$ by all the switches needed to obtain $\pi'_{i-1}[1,n]$ from $\pi'_{i-2}[1,n]$ that occur to the right of $p$. For $p+2d_1+1\leq i\leq 2(n+1)$, $\pi_i[1,n+1]$ is obtained from $\pi_{i-1}[1,n+1]$ by all the switches needed to obtain $\pi'_{i-2}[1,n]$ from $\pi'_{i-3}[1,n]$ that occur to the left of $p$, the transposition $(p,q_{i-2})$ (which is also needed to go from $\pi'_{i-3}[1,n]$ to $\pi'_{i-2}[1,n]$), the transposition $(p',q_{i-1})$, and all the switches needed to obtained $\pi'_{i-1}[1,n]$  from $\pi'_{i-2}[1,n]$ that occur to the right of $p$. 

Note that all the switches described above are actually possible because in the $(i-1)^{th}$ and $i^{th}$ moves of $\Pi'$ all the switches of $\Pi'$ occurring to the left of $p$  are independent of all the switches to the right of $p$. Moreover, each switch required in the half period $\{\pi'_0,\pi'_1,...,\pi'_{2n}\}$ of $\Pi'$, other than the first crossing switch that reverses $s'_1=(p,p+1,\dots,n,\overline{n},\dots,\overline{p+1},\overline{p})$, is used exactly once in $\Pi$; the first crossing switch of $\Pi'$ reversing $s'_1$ is replaced by the first crossing switch of $\Pi$ which reverses $s_1$ and so it takes care of all reversals of elements in $s'_1$ plus all reversals of $p'$ with each element of $s'_1$; and since $\{q_i:2\neq i\neq p\}=\{1,2,\dots,p-1\}$ and $\{q_i:p+2d_1\leq i\leq 2n\}=\{\overline{1},\overline{2},\dots,\overline{p-1}\}$, then $p'$ transposes exactly once with each of the points of $\Pi$ that are not in $s_1$. Therefore, $\Pi$ is a centrally symmetric allowable sequence. Moreover, $\Pi$ is in noncentral general position because all switches required by the half period $\{\pi_0,\pi_1,\dots,\pi_{2(n+1)}\}$ are transpositions except for $t$ crossing switches: all the crossing switches of $\Pi'$ except for the reversal of $s'_1$ which is replaced by the reversal of $s_1$. That is, $2(n+1)=d_1+d_2+\dots+d_t$ is the central signature of $\Pi$. Finally, since $\Pi$ has $2(n+1)$ points and $\{\pi_0,\pi_1,\dots,\pi_{2(n+1)}\}$ is a half period of $\Pi$ consisting of exactly $2(n+1)$ permutations plus the initial permutation $\pi_0$, then $\Pi$ is even-near-critical. 
\end{proof}
Theorem \ref{th:existence} shows the existence of an allowable sequence for any cyclic sequence $(d_1,d_2,\dots,d_t)$ and Theorem \ref{th:uniqueness} shows the uniqueness, up to combinatorial equivalence, of such a sequence proving the following theorem.

\begin{theorem}\label{th:main_pseudo}
Let $d=(d_1,...,d_t)$ with $t>1$ be a sequence of positive integers. Then, up to combinatorial equivalence, $d$ is the central signature of a unique even-near-critical centrally symmetric allowable sequence $\Pi$ in non-central general position.
\end{theorem}

Let $\text{DC}^{\,cs}_{ncgen}(d_1,d_2,\dots,d_t)$ be the unique allowable sequence guaranteed by Theorem \ref{th:main_pseudo} with $\pi_0=(12\dots n \,\overline n \dots\overline 2 \,\overline 1)$ and whose first move includes a crossing switch. The following result on the structure of $\text{DC}^{\,cs}_{ncgen}(d_1,d_2,\dots,d_t)$ is a direct consequence of Theorem \ref{th:uniqueness}.

\begin{corollary}\label{cor:struct_prop}
The allowable sequence $\Pi=\text{DC}^{\,cs}_{ncgen}(d_1,d_2,\dots,d_t)$ satisfies:
\begin{enumerate}
    \item \label{cor_part:positions}For $1\leq i\leq t$, $1\leq k\leq d_i$, and $1\leq j\leq n$, we have that 
    \begin{equation}\label{eq:positions}
        \pi_{\delta_i+j}(s_i(k))=
        \begin{cases}
         n-d_i+k & \text{if } j=0,\\
         n+d_i-k+1 & \text{if } 1\leq j\leq k,\\
         n+d_i-2k+j+1 & \text{if } k+1\leq j\leq k+n-d_i,\\
         2n-k+1 & \text{if } k+n-d_i\leq j\leq n+d_i-k+1,\\
         3n-2k+d_i+2& \text{if } n+d_i-k+1\leq j\leq 2n-k+1,\\
         n+d_i-k+1 & \text{if } 2n-k+1\leq j\leq 2n.
        \end{cases}
    \end{equation}
    \item \label{cor_part:central} The crossing switch of $s_i$ is part of move $\delta_i$.
    \item \label{cor_part:transp_mixed} The transpositions  $s_i(k)\overline{s_j(l)}$ and $\overline{s_i(k)}s_j(l)$ with $1\leq i<j\leq t$ are part of the move $(\delta_i+d_{i+1}+d_{i+2}+\dots+d_j+k-l)) \bmod 2n$. 
    \item \label{cor_part:transp} The transpositions $s_i(k)s_j(l)$ and $\overline{s_i(k)}\,\overline{s_j(l)}$ with $1\leq i<j\leq t$ are part of the move $(n+\delta_i+d_i+d_{i+1}+\dots+d_{j-1}+l-k)\bmod 2n$.
\end{enumerate}
\end{corollary}

\section{On the classification of geometric configurations}\label{sec:classification}

We start by defining a few families of geometric configurations introduced in \cite{JAM1984}. A \emph{bipencil} is a centrally symmetric configuration all but two of whose points are collinear. 
An \emph{exponential cross} is an affine copy of the following set of $2(s+t+2)+1$ points for some integers $s,t\ge 1, \lambda>1$:
\begin{equation*}
    \text{EX}_{\lambda}(s,t)=\left\{(0,0), (\pm \lambda^{i},0),(0,\pm \lambda^{j}): 0\le i \le s, 0\le j \le t\right\}.
\end{equation*}
A \emph{tricolumnar array} is an affine copy of the following set of $2(r+s+t)+3$ points for some integers $r,s,t\ge 2$:
\begin{equation*}
    \text{TC}(r,s;t)\!=\!
    \left\{(\pm1,k), \left(0,\tfrac{r}{2}\pm i\right),\left(0,\tfrac{r}{2}\pm\left(j-\tfrac{1}{2}\right)\right): 0\leq k \leq r, 0\leq i\leq s, j\in [t] \right\}.
\end{equation*}
Note that these configurations are centrally symmetric with an odd number of points. By removing their center, we obtain $\text{EX}^*_{\lambda}(s,t)$ and $\text{TC}^*(r,s;t)$, which extends these families to even configurations.  Jamison \cite{JAM1984} conjectured that, besides some sporadic configurations, any odd-critical configuration belongs to one of the following four infinite families: (1) even regular polygons with their center, (2) centrally symmetric bipencils, (3) exponential crosses, and (4) tricolumnar arrays.
\begin{figure}
    \centering
    \includegraphics[width=.75\linewidth]{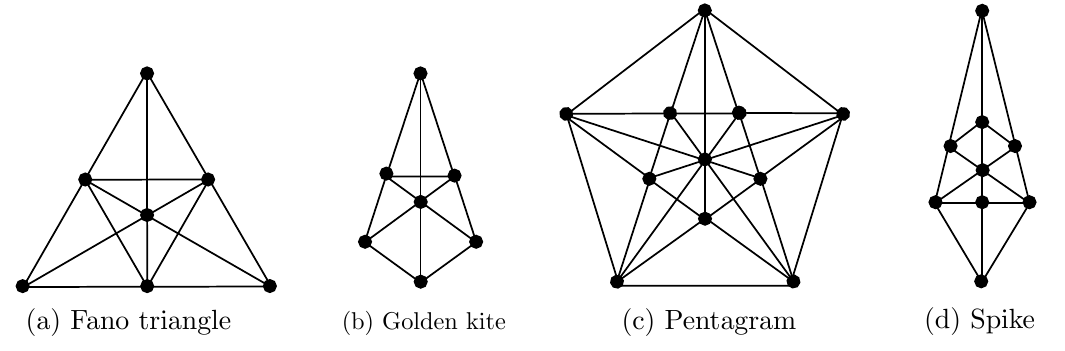}
    \caption{The four known odd-critical configurations that are not centrally symmetric. The coordinates for the spike are $x:0,\pm 1,\pm\sqrt{2}$ and $y:-2-\sqrt{2},-1,0,1/\sqrt{2},\sqrt{2},2(1+\sqrt{2})$. }
    \label{fig:non-centralsym}
\end{figure}
Furthermore, Jamison and Hill \cite{JAMH1983} provided a catalogue of $102$ sporadic odd-critical configurations which do not belong to any of the four infinite families above. Only 4 of them are not centrally symmetric, see Figure \ref{fig:non-centralsym}. Only one of the other sporadic configurations is in noncentral general position, see Figure \ref{fig:Z5_13_6}. Jamison conjectured that each odd-critical configuration of $n\ge 50$ points must belong to one of the four infinite families above. This conjecture remains open (see \cite{JAM1985} for a survey of similar open problems and conjectures) but there are some partial results supporting it. Here are some examples relevant to our work.
\begin{theorem}[Jamison \cite{JAM1984,JAM1986}]
\label{th:Jamison_uniqueness_results}
Let $P$ be a configuration of points. Then
\begin{enumerate}
    \item \label{th_part:Jamison_polygon}If $P$ is a near-critical configuration of points in general position, then $P$ is affinely equivalent to the set of vertices of a regular polygon. 
    \item \label{th_part:Jamison_exp_cross_unique} If $P$ is odd-critical and contained in two lines, then $P$ is affinely equivalent to a bipencil 
    or to an exponential cross. 
    \item \label{th_part:Jamison_tricolumn} If $P$ is odd-critical and contained in three parallel  lines, then $P$ is affinely equivalent to a bipencil or to a tricolumnar array. 
\end{enumerate}
\end{theorem}
The rest of the paper is dedicated to the geometric realizability  of the even-near-critical centrally symmetric allowable sequences in noncentral general position characterized by Theorem \ref{th:main_pseudo}. Our main result, Theorem \ref{th:noncentral_3}, completely classifies these configurations. We prove it in Section \ref{sec:proofs}, it follows from Theorems \ref{th:(>2,>1,...)}-\ref{th:(2,1,2)}.
\begin{theorem}\label{th:noncentral_3}
The allowable sequence $\text{DC}^{\,cs}_{ncgen}(d_1,d_2,\dots,d_t)$ is geometrically realizable, if and only if,
\begin{enumerate}
    \item $d_1=d_2=\cdots=d_t=1$,
    \item $t=2$,
    \item $t=3$ and two entries of $(d_1,d_2,d_3)$ are equal to 1, or 
    \item $t=3$ and $(d_1,d_2,d_3)=(2,2,2)$.
\end{enumerate}
\end{theorem}
Theorem \ref{th:noncentral_3} allows us to prove Jamison's conjecture when restricted to centrally symmetric configurations in non-central general position, and in fact extend it to even configurations in this case. 
\begin{theorem}\label{th:noncentral_2}
Any centrally symmetric odd-critical or even-near-critical set of points in noncentral general position is affinely equivalent to one of the following sets, with or without its center of symmetry:
\begin{enumerate}
    \item the set of vertices of a regular polygon with an even number of sides,
    \item a centrally symmetric bipencil, 
    \item an exponential cross,
    \item a tricolumnar array with $r=1$ and $t-s=0$ or 1.
    \item the configurations $(\text{Z}\,5,12,6)$ and $(\text{Z}\,5,13,6)$ in Figure \ref{fig:Z5_13_6}.
\end{enumerate}
\end{theorem}
\begin{figure}[h]
    \centering
    \includegraphics[width=.24\linewidth]{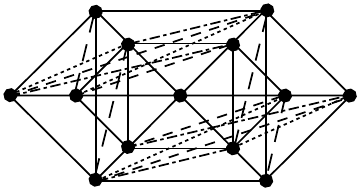}\hspace{.5in } \includegraphics[width=.24\linewidth]{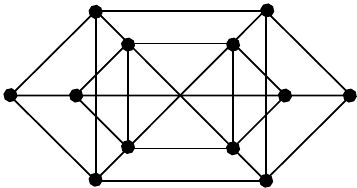}
    \caption{The direction-(near)-critical configurations $(\text{Z}5,13,6)$ and $(\text{Z}5,12,6)$. Their coordinates are $x:0,\pm 1$,$\pm \tau$,$\pm 2$,$\pm 2\tau $ and $y:0,\pm 1,\pm \tau$, where $\tau=(1+\sqrt{5})/2$ is the golden ratio. }
    \label{fig:Z5_13_6}
\end{figure}
\begin{figure}
    \centering
    \includegraphics{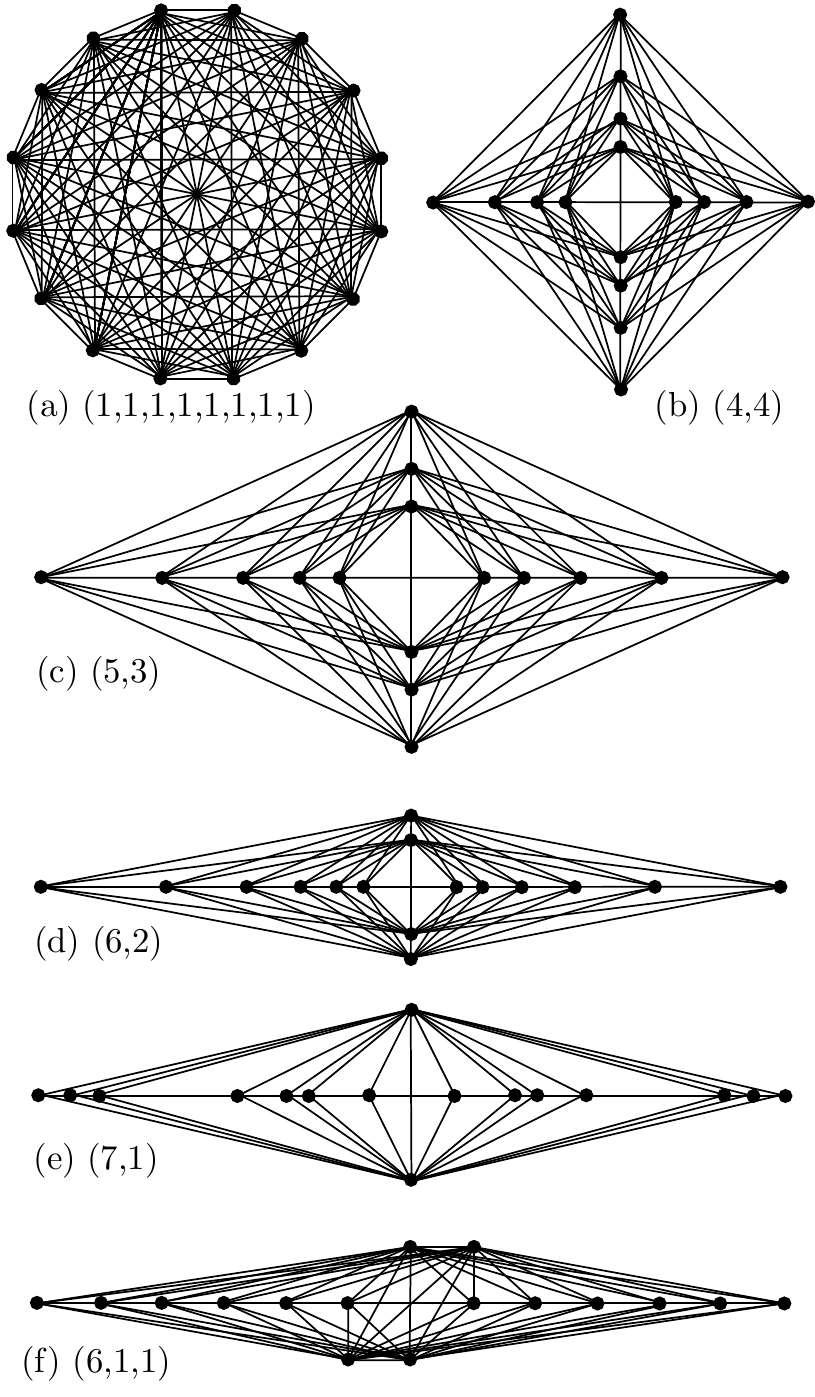}
    \caption{Even-near-critical centrally symmetric geometric realizations with 16 points: (a) regular polygon, (b-d) exponential crosses, (e) bipencil, and (f) tricolumnar array. Their circular sequences are included in the Appendix.}
    \label{fig:geomrealizations}
\end{figure}
\begin{proof}
Let $P$ be a centrally symmetric odd-critical or even-near-critical set of points in noncentral general position. Since adding or removing the center of symmetry from $P$ preserves the central symmetry, criticality, and noncentral general position; we assume that $P$ has an even number of points. Then the circular sequence of $P$ is combinatorially equivalent to $\Pi=\text{DC}^{\,cs}_{ncgen}(d_1,d_2,\dots,d_t)$ for some of the cases listed in Theorem \ref{th:noncentral_2}.

If $d_1=d_2=\dots=d_t=1$, then $\Pi$ corresponds to a configuration in general position and so, by Theorem \ref{th:Jamison_uniqueness_results}(\ref{th_part:Jamison_polygon}), $P$ is affinely equivalent to the set of vertices of a regular polygon with an even number of sides; see Figure \ref{fig:geomrealizations}(a). If $t=2$, then $\Pi$ corresponds to a configuration contained in two lines and so $P$ is a centrally symmetric bipencil when $d_2=1$, see Figure \ref{fig:geomrealizations}(e); or by Theorem \ref{th:Jamison_uniqueness_results}(\ref{th_part:Jamison_exp_cross_unique}) $P$ (technically $P$ plus its center) is affinely equivalent to an exponential cross when $d_2\ge 2$; see Figures \ref{fig:geomrealizations}(b-d). The allowable sequence $\text{DC}^{\,cs}_{ncgen}(d_1,1,1)$ is geometrically realized by the tricolumnar array $\text{TC}^*(1,\lfloor d_1/2\rfloor,\lceil d_1/2\rceil)$, see Figure \ref{fig:geomrealizations}(d). Among other things, Theorem \ref{th:(>1,1,1...)} proves that any geometric realization is affinely equivalent to this tricolumnar array. The uniqueness up to affine equivalence also follows from Theorem \ref{th:Jamison_uniqueness_results}(\ref{th_part:Jamison_tricolumn}) by realizing that the first move of $\text{DC}^{\,cs}_{ncgen}(d_1,1,1)$ consists of three switches reversing all points. Finally, $\text{DC}^{\,cs}_{ncgen}(2,2,2)$ is combinatorially equivalent to the circular sequence of the sporadic configuration $(\text{Z}5,12,6)$ (defined in \cite{JAMH1983} and shown in Figure \ref{fig:Z5_13_6}) without its center, which we denote by $(\text{Z}5,12,6)$. In Theorem \ref{th:(2,2,...,2)}, we show that any geometric realization of $\text{DC}^{\,cs}_{ncgen}(2,2,2)$ is affinely equivalent to $(\text{Z}5,12,6)$.
\end{proof}

\subsection{Geometric realizations}\label{sec:proofs}
Theorem \ref{th:noncentral_3} follows from Theorems \ref{th:(>2,>1,...)}-\ref{th:(2,1,2)} in this section. We start with a result for centrally symmetric even-near-critical allowable sequences in noncentral general position.

\begin{proposition}\label{prop:extreme_expcross} 
Let $\Pi$ be an even-near-critical centrally symmetric allowable sequence of $n$ points in noncentral general position equipped with the central signature $(d_1,d_2\dots,d_t)$. Then
\begin{enumerate}
    \item \label{prop_part:extreme} The extreme points of each crossing line are extreme points of the configuration.
    \item \label{prop_part:expcross}The sequence $\Pi|_{s_i\cup s_{i+1}}$ is combinatorially equivalent to $\text{DC}^{\,cs}_{ncgen}(d_i,d_{i+1})$.
\end{enumerate}
\end{proposition}
\begin{proof}

\noindent (1) We need to check that, for $1\leq i\leq t$, the end points $s_i(1)$ and $\overline{s_i(1)}=s_i(2d_1)$ of the crossing substring $s_i$ visit the first or last position in a permutation of $\Pi$. By Identity \ref{eq:positions}, $\pi_{\delta_i+n \bmod 2n}(s_i(1))=2n$ and $\pi_{\delta_i+n \bmod 2n}(\overline{s_i(1)})=1$.

\noindent (2) The central symmetry and noncentral general position of $\Pi|_{s_i\cup s_{i+1}}$ are clearly inherited from $\Pi$. We only need to show that $\Pi|_{s_i\cup s_j}$ is even-near-critical, that is, it has a halfperiod of length $2(d_i+d_{i+1})$. This is equivalent to proving that there are $2(d_i+d_j)$ moves of $\Pi$ that involve points of $s_i\cup s_j$. 

Suppose without loss of generality that $i<j$. The string $s_i$ is switched in the move $\delta_i$  and $s_j$ is switched in the move $\delta_j$. By Corollary \ref{cor:struct_prop}(\ref{cor_part:transp_mixed}), each transposition involving a point of $s_i[1,d_i]$ and a point of $s_j[d_j+1,2d_j]$, or a point of $s_i[d_i+1,2d_i]$ and a point of $s_j[1,d_j]$, is part of the move $(\delta_i+d_2+d_3+\dots+d_{j-1}-1+m) \bmod 2n$ for some  $2\leq m\leq d_i+d_j$. This is because for $1\leq k\leq d_i$ and $1\leq l\leq d_j$, we have that $\delta_i+d_2+d_3+\dots+d_j+k-l=\delta_i+d_2+d_3+\dots+d_{j-1}-1+k+(d_j+1-l)$ and $k+(d_j+1-l)$ can take any value between 2 and $d_i+d_j$.  Similarly,  Corollary \ref{cor:struct_prop}(\ref{cor_part:transp}) implies that each transposition involving a point of $s_i[1,d_i]$ and a point of $s_j[1,d_j]$, or a point of $s_i[d_i+1,2d_i]$ and a point of $s_j[d_j+1,2d_j]$, is part of the move $(n+\delta_i+d_1+d_2+\dots+d_{j-1}+m) \bmod 2n$ for some  $2\leq m\leq d_i+d_j$. This time $n+\delta_i+d_1+d_2+\dots+d_{j-1}+l-k=n+\delta_i+d_2+\dots+d_{j-1}+(d_i+1-k)+l$ and $(d_i+1-k)+l$ can take any value between $2$ and $d_i+d_j$ when $1\leq k\leq d_i$ and $1\leq l \leq d_j$. Therefore,  all switches involving only points of $s_i\cup s_j$ are part of the $2(d_i+d_j)$ moves $\delta_i,\delta_j,(\delta_i+d_2+d_3+\dots+d_{j-1}-1+m) \bmod 2n,(n+\delta_i+d_2+d_3+\dots+d_{j-1}-1+m) \bmod 2n$ for $2\leq m\leq d_i+d_j$.
\end{proof}

\begin{corollary}\label{cor:bip_expcross}
Suppose that the allowable sequence $\Pi=\text{DC}^{\,cs}_{ncgen}(d_1,d_2,\dots,d_t)$ with $d_i\geq 2$ and $d_j\geq 1$ for some distinct subindices $i,j\in[t]$ is  geometrically realizable. Then in any geometric realization of $\Pi$, the set of points corresponding to $\Pi|_{s_i\cup s_j}$ is a centrally symmetric bipencil if $d_j=1$
and affinely equivalent to the exponential cross $\text{EX}^*_\lambda(d_i-1,d_j-1)$ if $d_j\geq 2$.
\end{corollary}
\begin{proof}
Suppose that $P$ is a geometric realization of $\Pi$. By Proposition \ref{prop:extreme_expcross}(\ref{prop_part:expcross}), the induced allowable sequence $\Pi|_{s_i\cup s_j}$ is also even-near-critical. Note that all points of $\Pi|_{s_i\cup s_j}$ are contained in two lines: the line containing the string $s_i$ (with $2d_i$ points) and the line containing $s_j$ (with $2d_j$ points). Then, by Theorem \ref{th:Jamison_uniqueness_results}(\ref{th_part:Jamison_exp_cross_unique}), the subset of $P$ corresponding to $\Pi|_{s_i\cup s_j}$ is a centrally symmetric bipencil of $2d_i+2$ points when $d_j=1$, and affinely equivalent to the exponential cross $\text{EX}^*_\lambda(d_i-1,d_j-1)$ for some $\lambda>1$ with $2(d_i+d_j)$ points when $d_j\geq 2$.
\end{proof}

In the following proofs, $\Pi$ is an allowable sequence with central signature $\\(d_1,d_2,\cdots,d_t)$ and $P$ is a potential geometric realization of $\Pi$. We use the same labels for the points of $\Pi$ and their realizations in $P$. Also, the subindices of the $d_i$s are taken modulo $t$.

\begin{theorem}\label{th:(>2,>1,...)}
The allowable sequence $\Pi=\text{DC}^{\,cs}_{ncgen}(d_1,d_2,\dots,d_t)$ with $t\geq 3$, $d_i\geq 3$, and $d_j\geq 2$ for some distinct subindices $i,j\in[t]$ is  not geometrically realizable. 
\end{theorem}

\begin{proof}
Without loss of generality, assume that $d_1\ge 3$ and $d_i\geq d_j \ge 2$ for some $2\leq j< t$. Suppose by contradiction that there is a geometric realization $P$ of $\Pi$. By Corollary \ref{cor:bip_expcross} (and applying an affine transformation to $P$ if necessary), we can assume that the subset of $P$ corresponding to $\Pi|_{s_1\cup s_j}$ is precisely  $\text{EX}^*_\lambda(d_1-1,d_j-1)$ for some $\lambda>1$, that is,
\begin{align*}
    & \ell_1 ~:~ s_1(k)=(0, \lambda^{d_1-k}) \text{ and }\overline{s_1(k)}=(0, -\lambda^{d_1-k})~~\text{for } 0\leq k< d_1\\
    & \ell_j ~:~ s_j(l)=( -\lambda^{d_j-l},0)\text{ and } \overline{s_j(l)}=( \lambda^{d_j-l},0)~~\text{for } 0\leq l< d_j.
\end{align*}
By Corollary \ref{cor:struct_prop}(\ref{cor_part:transp_mixed}), the transposition $\overline{s_1(1)}s_{j+1}(d_{j+1})$ occurs in move $\delta_1+d_2+d_3+\dots+d_{j+1}+1-d_{j+1}=d_2+d_3+\dots+d_j+1$ and the transposition $\overline{s_1(2)}s_j(1)$ occurs in move $\delta_1+d_2+d_3+\dots+d_j+1-1=d_2+d_3+\dots+d_j+1$. Thus the lines $\overline{s_1(1)}s_{j+1}(d_{j+1})$ and $\overline{s_1(2)}s_j(1)$ are parallel. Similarly, the transposition $\overline{s_1(2)}s_{j+1}(d_{j+1})$ occurs in move $\delta_1+d_2+d_3+\dots+d_{j+1}+2-d_{j+1}=d_2+d_3+\dots+d_j+2$ and the transposition $\overline{s_1(3)}s_j(1)$ occurs in move $\delta_1+d_2+d_3+\dots+d_j+3-1=d_2+d_3+\dots+d_j+2$. Thus the lines $\overline{s_1(2)}s_{j+1}(d_{j+1})$ and $\overline{s_1(3)}s_j(1)$ are parallel. 

Since $s_{j+1}$ is reversed after $s_j$, $s_{j+1}$ is in the third quadrant and so it has coordinates $(-x,-y)$ for some $x,y>0$. Equaling the slopes of each pair of parallel lines, we have

\begin{equation*}
    \frac{\lambda^{d_1-2}}{\lambda^{d_j-1}}=\frac{\lambda^{d_1-1}-y}{x} ~~ \text{ and } ~~ \frac{\lambda^{d_1-3}}{\lambda^{d_j-1}}=\frac{\lambda^{d_1-2}-y}{x}.
\end{equation*}
Multiplying the second identity by $\lambda$, the identities imply that $y=y\lambda$. This is impossible because $y>0$ and $\lambda>1$.
\end{proof}

\begin{theorem}\label{th:(>1,1,1...)}
The allowable sequence $\Pi=\text{DC}^{\,cs}_{ncgen}(d_1,d_2,\dots,d_t)$ with $t\geq 3$, $d_i\geq 2$, and $d_{i+1}=d_{i+2}=1$ for some $i\in [t]$ is  geometrically realizable if and only if $t=3$. Furthermore, any realization of $\Pi$ is affinely equivalent to $\text{TC}^{\,*}(1,\lfloor d_i/2\rfloor,\lceil d_i/2\rceil)$.
\end{theorem}

\begin{proof}

Without loss of generality, assume that $d_1\ge 2$ and $d_2=d_3=1$. Suppose by contradiction that there is a geometric realization $P$ of $\Pi$. By Corollary \ref{cor:bip_expcross}, we can assume that the subset of $P$ corresponding to $\Pi|_{s_1\cup s_2}$ is the centrally symmetric bipencil
\begin{align*}
    & \ell_1 ~:~ s_1(d_1+1-k)=(0, \lambda_k) \text{ and } ~\overline{s_1(d_1+1-k)}=(0, -\lambda_k)\\& \hspace{.5in}\text{for } k\in[d_1], 1=\lambda_1<\lambda_2<\dots < \lambda_{d_1},\\
    & \ell_2 ~:~ s_2(1)=(-1,0)\text{ and }  \overline{s_2(1)}=(1,0).
\end{align*}

Consider $k\in[d_1]$. By Corollary \ref{cor:struct_prop}(\ref{cor_part:transp}), the transposition $s_1(d_1+1-k)s_3(1)$ occurs in move $n+\delta_1+d_1+d_2+1-d_1-1+k=n+1+k$, and the transposition $s_1(d_1-k)s_2(1)$ occurs in move $n+\delta_1+d_1+1-d_1+k=n+1+k$. Thus the lines $s_1(d_1+1-k)s_3(1)$ and $s_1(d_1-k)s_2(1)$ are parallel. By Corollary \ref{cor:struct_prop}(\ref{cor_part:transp_mixed}), the transposition $\overline{s_1(d_1-k-1)}s_3(1)$ occurs in move $\delta_1+d_2+d_3+d_1-k-1-1=d_1-k$ and the transposition $\overline{s_1(d_1-k)}s_2(1)$ occurs in move $\delta_1+d_2+d_1-k-1=d_1-k$. Thus the lines $\overline{s_1(d_1-k-1)}s_3(1)$  and $\overline{s_1(d_1-k)}s_2(1)$ are parallel. Similarly, the transposition  $\overline{s_1(d_1)}s_3(1)$ occurs in move $\delta_1+d_2+d_3+d_1-1=d_1+1$. Also, by Corollary \ref{cor:struct_prop}(\ref{cor_part:central}), $s_2$ reverses in move $\delta_2=d_1+d_2=d_1+1$. 
Thus the lines $\overline{s_1(d_1)}s_3(1)$ and $\ell_2$ are parallel. This means that $s_3(1)=(-a,-1)$. Moreover, since $s_3$ is reversed after $s_1$ and $s_2$, then $s_3(1)$ is in the third quadrant and so $a>0$.
Equaling the slopes of each pair of parallel lines, we obtain
\begin{equation*}
    \frac{\lambda_k+1}{a}=\frac{\lambda_{k+1}}{1} ~~ \text{ and } ~~ \frac{1-\lambda_{k+1}}{a}=\frac{-\lambda_k}{1}.
\end{equation*}
This means that $a=(\lambda_k+1)/\lambda_{k+1}=(\lambda_{k+1}-1)/\lambda_k$ for any $k\in[d_1]$ with $\lambda_1=1$. Solving this recursion gives $a=1$ and $\lambda_k=k$ for any $k\in[d_1]$.
Note that the subconfiguration of $2(d_1+2)$ points determined by $s_1,s_2$, and $s_3$ corresponds to the points
\begin{center}
\begin{align*}
    & \ell_1 ~:~ (0, \pm k) ~~ \text{ for }~~ k\in[d_1],\\
    & \ell_2 ~:~ (\pm1,0),\\
    & \ell_3 ~:~ \pm(1,1).
\end{align*}
\end{center}
This configuration is actually direction-critical. It is similar to the tricolumnar array $\text{TC}^*(\lfloor d_1/2\rfloor,\lceil d_1/2\rceil)$, see Figure \ref{fig:geomrealizations}. This means that any geometric realization of $\text{DC}^{\,cs}_{ncgen}(d_1,1,1)$ with $d_1\geq 2$ is affinely equivalent to $\text{TC}^*(1,\lfloor d_1/2\rfloor,\lceil d_1/2\rceil)$.

Now assume by contradiction that $t\geq 4$ and consider the point $s_4(d_4)$. By Corollary \ref{cor:struct_prop}(\ref{cor_part:central}), $s_3$ reverses in move $\delta_3=d_1+2d_2+d_3=d_1+3$. By Corollary \ref{cor:struct_prop}(\ref{cor_part:transp_mixed}), the transposition $\overline{s_3(1)}s_4(d_4)$ occurs in move $\delta_3+d_4+1-d_4=d_1+2d_2+d_3+1=d_1+3$. Thus the lines $\ell_3$ and $\overline{s_3(1)}s_4(d_4)$ are parallel. But $\overline{s_3(1)}$ is on $\ell_3$ and so the lines are actually equal. This means that $s_4(d_4)\in\ell_3$, which is impossible.
\end{proof}

\begin{theorem}\label{th:(2,2,...,2)}
The allowable sequence $\Pi=\text{DC}^{\,cs}_{ncgen}(d_1,d_2,\dots,d_t)$ with $t\geq 3$ and $d_i=d_{i+1}=d_{i+2}=2$ for some $i\in[t]$ is geometrically realizable if and only if $t=3$. Moreover, any realization of $\text{DC}^{\,cs}_{ncgen}(2,2,2)$ is affinely equivalent to $(\text{Z}\,5,12,6)$.
\end{theorem}
\begin{proof}
Without loss of generality, assume that $d_1=d_2=d_3=2$. Suppose that there is a geometric realization $P$ of $\Pi$. By Corollary \ref{cor:bip_expcross}, we can assume that the subset of $P$ corresponding to $\Pi|_{s_1\cup s_2}$ is the exponential cross
\begin{align*}
    & \ell_1 ~:~ s_1(1)=(0, \lambda),s_1(2)=(0,1) ,~\overline{s_1(2)}=(0,-1), ~\overline{s_1(1)}=(0, -\lambda),\\
    & \ell_2 ~:~ s_2(1)=( \lambda,0),s_2(2)=(1,0) ,~\overline{s_2(2)}=(-1,0), ~\overline{s_2(1)}=(-\lambda,0).
\end{align*}
By Corollary \ref{cor:struct_prop}(\ref{cor_part:transp}), the transposition $s_1(1)s_2(2)$ occurs in move $n+\delta_1+d_1+2-1=n+3$ and the transposition $s_1(2)s_3(1)$ occurs in move $n+\delta_1+d_1+d_2+1-2=n+3$. Thus the lines $s_1(1)s_2(2)$ and $s_1(2)s_3(1)$ are parallel. Similarly, the transposition $s_1(1)s_3(1)$ occurs in move $n+\delta_1+d_1+d_2+1-1=n+4$ and the transposition $s_1(2)s_3(2)$ occurs in move $n+\delta_1+d_1+d_2+2-2=n+4$. Thus the lines $s_1(1)s_3(1)$ and $s_1(2)s_3(2)$ are parallel. By Corollary \ref{cor:struct_prop}(\ref{cor_part:transp_mixed}), the transposition $\overline{s_1(2)}s_2(1)$ occurs in move $\delta_1+d_2+2-1=3$ and the transposition $\overline{s_1(1)}s_3(2)$ occurs in move $\delta_1+d_2+d_3+1-2=3$. Thus the lines $\overline{s_1(2)}s_2(1)$ and $\overline{s_1(1)}s_3(2)$ are parallel. Similarly, the transposition $\overline{s_1(2)}s_3(2)$ occurs in move $\delta_1+d_2+d_3+2-2=4$ and the transposition $\overline{s_1(1)}s_3(1)$ occurs in move $\delta_1+d_2+d_3+1-1=4$. Also, by Corollary \ref{cor:struct_prop}(\ref{cor_part:central}), $s_2$ reverses in move $\delta_2=d_1+d_2=4$. 
Thus the lines $\overline{s_1(2)}s_3(2)$, $\overline{s_1(1)}s_3(1)$, and $\ell_2$ are parallel. Thus $s_3(1)=(-a,-\lambda)$ and $s_3(2)=(-b,-1)$ for some $a$ and $b$. Moreover, since $s_3$ is reversed after $s_2$, $s_3(1)$ and $s_3(2)$ are in the third quadrant and so $a,b>0$.
Equaling the slopes of each pair of parallel lines, we have
\begin{equation*}
    \frac{\lambda}{1}=\frac{\lambda+1}{a} ~~, ~~ \frac{\lambda}{1}=\frac{b}{\lambda-1} ~~ \text{ , and } ~~ \frac{2\lambda}{a}=\frac{2}{b}.
\end{equation*}
Hence, $\lambda=a/b=(\lambda+1)(\lambda-1)/\lambda^2$, that is, $\lambda=(1+\sqrt{5})/2,a=\lambda,$ and $b=1$. Then $s_3(1)=(-\lambda,-\lambda)$, $s_3(2)=(-1,-1)$, $\overline{s_3(1)}=(\lambda,\lambda)$ and $\overline{s_3(2)}=(1,1)$. Note that the subconfiguration of $12$ points determined by $s_1,s_2$, and $s_3$ corresponds to the points

\begin{align*}
    & \ell_1 ~:~ \pm\left(0, \frac{1+\sqrt{5}}{2}\right),\pm(0,1),\\
    & \ell_2 ~:~ \pm\left( \frac{1+\sqrt{5}}{2},0\right),\pm(1,0),\\
    & \ell_3 ~:~ \pm\left(\frac{1+\sqrt{5}}{2},\frac{1+\sqrt{5}}{2}\right), \pm(1,1).
\end{align*}
This configuration is actually direction-critical. It is a similar copy of $(\text{Z}5,12,6)$ shown in Figure \ref{fig:Z5_13_6}, which means that any realization of $\text{DC}^{\,cs}_{ncgen}(2,2,2)$ is affinely equivalent to $(\text{Z}5,12,6)$.

Now assume by contradiction that $t\geq 4$ and consider the point $s_4(d_4)$. By Corollary \ref{cor:struct_prop}(\ref{cor_part:central}), $s_3$ reverses in move $\delta_3=d_1+2d_2+d_3=8$. By Corollary \ref{cor:struct_prop}(\ref{cor_part:transp_mixed}), the transposition $\overline{s_2(2)}s_4(d_4)$ occurs in move $\delta_2+d_3+d_4+2-d_4=d_1+2d_2+d_3+2=8$. Thus the lines $\ell_3$ and $\overline{s_2(2)}s_4(d_4)$ are parallel. Since $\ell_3$ has slope $1$ and the line $\overline{s_2(2)}s_4(d_4)$ passes through $\overline{s_2(2)}=(1,0)$, then $s_4(d_4)=(-c,-1-c)$ and $c>0$ because $s_4(d_4)$ must be in the third quadrant due to $s_4$ reversing after $s_1$ and $s_2$. By Corollary \ref{cor:struct_prop}(\ref{cor_part:transp_mixed}), the transposition $\overline{s_2(2)}s_3(1)$  occurs in move $\delta_2+d_3+2-1=d_1+d_2+d_3-1=5$, and the transposition $\overline{s_2(1)}s_4(d_4)$  occurs in move $\delta_2+d_3+d_4+1-d_4=d_1+d_2+d_3-1=5$. Thus the lines $\overline{s_2(2)}s_3(1)$ and $\overline{s_2(1)}s_4(d_4)$ are parallel and so $\lambda/(1+\lambda)=(1+c)/(\lambda+c)$. Since $\lambda=(1+\sqrt{5})/2$, then $c=\lambda^2-\lambda-1=0$, getting a contradiction.
\end{proof}

\begin{theorem}\label{th:(2,1,2)}
The allowable sequence $\Pi=\text{DC}^{\,cs}_{ncgen}(d_1,d_2,\dots,d_t)$ with $t\geq 3$ and $d_i=d_{i+2}=2$ and $d_{i+1}=1$ for some $i\in[t]$ is not geometrically realizable. 
\end{theorem}
\begin{proof}
Without loss of generality, assume that $d_1=d_3=2$ and $d_1=1$. Suppose that there is a geometric realization $P$ of $\Pi$. By Corollary \ref{cor:bip_expcross}, we can assume that the subset of $P$ corresponding to $\Pi|_{s_1\cup s_3}$ is the exponential cross
\begin{align*}
    & \ell_1 ~:~ s_1(1)=(0, \lambda),s_1(2)=(0,1) ,~\overline{s_1(2)}=(0,-1), ~\overline{s_1(1)}=(0, -\lambda),\\
    & \ell_3~:~ s_3(1)=( \lambda,0),s_3(2)=(1,0) ,~\overline{s_3(2)}=(-1,0), ~\overline{s_3(1)}=(-\lambda,0).
\end{align*}
By Corollary \ref{cor:struct_prop}(\ref{cor_part:transp}), the transposition $s_1(2)s_3(1)$ occurs in move $n+\delta_1+d_1+d_2+1-2=n+2$ and the transposition $s_1(1)s_2(1)$ occurs in move $n+\delta_1+d_1+1-1=n+2$. Thus the lines $s_1(2)s_3(1)$ and $s_1(1)s_2(1)$ are parallel. By Corollary \ref{cor:struct_prop}(\ref{cor_part:transp_mixed}), the transposition $\overline{s_1(2)}s_3(1)$ occurs in move $\delta_1+d_2+d_3+2-1=4$ and the transposition $\overline{s_2(1)}s_3(2)$ occurs in move $\delta_2+d_3+1-2=d_1+d_2+d_3-1=4$. Thus the lines $\overline{s_1(2)}s_3(1)$ and $\overline{s_2(1)}s_3(2)$ are parallel. Similarly, the transposition $\overline{s_1(2)}s_3(2)$ occurs in move $\delta_1+d_2+d_3+2-2=3$. Also, by Corollary \ref{cor:struct_prop}(\ref{cor_part:central}), $s_2$ reverses in move $\delta_2=d_1+d_2=3$. 
Thus the lines $\overline{s_1(2)}s_3(2)$ and $\ell_2$ are parallel, both with slope $-1$. Thus $s_2(1)=(-a,a)$ and since $s_2$ is reversed after $s_1$ and before $s_3$, then $s_3(1)$ is in the second quadrant and so $a>0$.
Equaling the slopes of each pair of parallel lines, we have
\begin{equation*}
    \frac{1}{\lambda}=\frac{\lambda-a}{a} ~~
    ~~ \text{ and } ~~ \frac{-1}{\lambda}=\frac{-a}{a+1}.
\end{equation*}
Hence, $a=1/(\lambda-1)=\lambda^2/(\lambda+1)$. This has a unique real solution for $\lambda$, and this solution satisfies $1<\lambda <2$. 
We separately analyze the cases $t=3$ and $t\geq 4$. When $t=3$, the transposition $s_3(2)s_2(1)$ occurs in move $n+\delta_2+d_2+2-1=n+d_1+2d_2+1=5+2+2+1=10\equiv 0 \pmod{10}$ by Corollary \ref{cor:struct_prop}(\ref{cor_part:transp}). This means that $s_3(2)s_2(1)$ is parallel to $\ell_1$, which is vertical. Since $s_3(2)=(-1,0)$ and $s_2(1)=(-a,a)$, then $a=1$. But $a=1/(\lambda-1)=1$ implies $\lambda=2$ contradicting $1<\lambda <2$. If $t\geq 4$, consider the point $s_4(d_4)$. By Corollary \ref{cor:struct_prop}(\ref{cor_part:central}), $s_3$ reverses in move $\delta_3=d_1+2d_2+d_3=6$. By Corollary \ref{cor:struct_prop}(\ref{cor_part:transp_mixed}), the transposition $\overline{s_2(1)}s_4(d_4)$ occurs in move $\delta_2+d_3+d_4+1-d_4=d_1+d_2+d_3+1=6$. Thus the lines $\ell_3$ and $\overline{s_2(1)}s_4(d_4)$ are parallel. Then $\overline{s_2(1)}s_4(d_4)$ is horizontal with $\overline{s_2(1)}=(a,-a)$. So $s_4(d_4)=(b,-a)$ for some $b>0$ as it should be in the 4th quadrant due to $s_4$ reversing after $s_1$ and $s_3$. By Corollary \ref{cor:struct_prop}(\ref{cor_part:transp_mixed}), the transposition $\overline{s_1(1)}s_4(d_4)$  occurs in move $\delta_1+d_2+d_3+d_4+1-d_4=d_2+d_3+1=4$ and so it is parallel to the line $\overline{s_1(2)}s_3(1)$ (see above). Thus $-1/\lambda=a/(c-\lambda)$. This implies $c=\lambda(\lambda-2)/(\lambda-1)$, which is negative because $1<\lambda<2$, getting a contradiction.
\end{proof}
We are finally ready to prove Theorem \ref{th:noncentral_3}. \begin{proof}[Proof of Theorem \ref{th:noncentral_3}]
Let $d=(d_1,d_2,\dots,d_t)$ and $\Pi=\text{DC}^{\,cs}_{ncgen}(d)$. When all entries of $d$ are 1s, $\Pi$ is geometrically realizable by the regular polygon with $2t$ sides. When $t=2$, $\Pi$ is geometrically realizable by an exponential cross if $d_1,d_2\geq 2$ or by a centrally symmetric bipencil otherwise. Suppose that  $\Pi$ is geometrically realizable for some $t\geq 3$ and some $d_i\geq 2$. Without loss of generality, assume that $d_1\geq 2$ is the largest entry of $d$. If $d_1\geq 3$, then Theorem \ref{th:(>2,>1,...)} implies that $d_2=d_3=\dots =d_t=1$. By Theorem \ref{th:(>1,1,1...)}, $t=3$ and so $d=(d_1,1,1)$. Moreover, any realization of $\Pi$ is affinely equivalent to $\text{TC}^*(1,\lfloor d_1/2\rfloor,\lceil d_1/2\rceil)$.

Now assume that each entry of $d$ is 1 or 2. By Theorem \ref{th:(2,2,...,2)}, either  $d=(2,2,2)$, which is realized by $(\text{Z}5,12,6)$; or there are no more than two consecutive 2s in $d$. By Theorem \ref{th:(>1,1,1...)}, either $d=(2,1,1)$, which is realized by the tricolumnar array $\text{TC}^*(1,1,1)$; or there are no consecutive 1s in $d$. Thus $(2,1,2)$ must be a substring of consecutive elements of $d$. But by Theorem \ref{th:(2,1,2)}, $\Pi$ is not geometrically realizable in this case.
\end{proof}

\section{Future Work}
Our classification results in Theorems \ref{th:noncentral_3} and \ref{th:noncentral_2} imply that any other odd-critical or even-near-critical centrally symmetric configuration must contain a connecting line of three or more points not passing through its center. Even though this does not completely settles Jamison's conjectures on large enough direction-critical configurations \cite{JAMH1983,JAM1985}, we hope that our techniques bring us a step closer to understanding the structure of all direction-(near)-critical configurations that are centrally symmetric. A structural result for even-near-critical allowable sequences similar to Theorem \ref{th:uniqueness}, when the noncentral general position hypothesis is removed, is still needed.

\begin{figure}[htb]
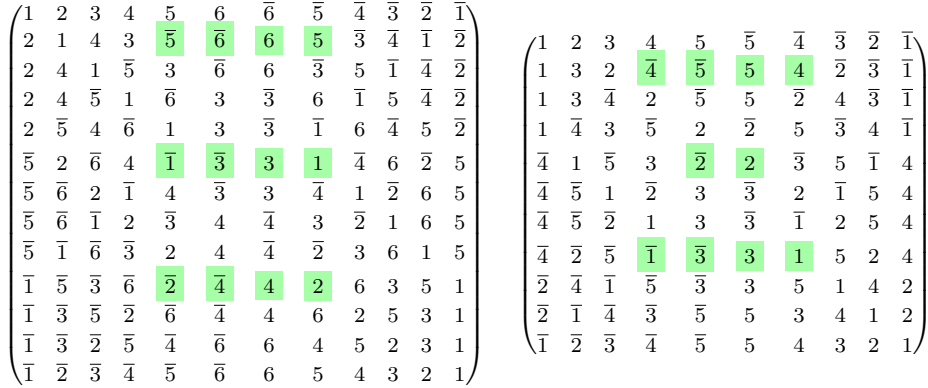

    \centering
\begin{footnotesize}
\begin{equation*}
    \setlength\arraycolsep{4.2pt}
    \renewcommand{\arraystretch}{0.2}
    \begin{pmatrix}
    1 & 2 & 3 & 4 & 5 & 6 & \overline{6} & \overline{5} & \overline{4} & \overline{3} & \overline{2} & \overline{1}\\
    2 & 1 & 4 & 3 & \colorbox{green!35}{$\overline{5}$} & \colorbox{green!35}{$\overline{6}$} & \colorbox{green!35}{$6$} & \colorbox{green!35}{$5$} & \overline{3} & \overline{4} & \overline{1} & \overline{2}\\
    2 & 4 & 1 & \overline{5} & 3 & \overline{6} & \colorbox{white!20}6 & \overline{3} & 5 & \overline{1} & \overline{4} & \overline{2}\\
    2 & 4 & \overline{5} & 1 & \overline{6} & \colorbox{white!20}3 & \overline{3} & 6 & \overline{1} & 5 & \overline{4} & \overline{2}\\
    2 & \overline{5} & 4 & \overline{6} & 1 & \colorbox{white!20}3 & \overline{3} & \overline{1} & 6 & \overline{4} & 5 & \overline{2}\\
    \overline{5} & 2 & \overline{6} & 4 & \colorbox{green!35}{$\overline{1}$} & \colorbox{green!35}{$\overline{3}$} & \colorbox{green!35}{$3$} & \colorbox{green!35}{$1$} & \overline{4} & 6 & \overline{2} & 5 \\
    \overline{5} & \overline{6} & 2 & \overline{1} & \colorbox{white!20}4 & \overline{3} & 3 & \overline{4} & 1 & \overline{2} & 6 & 5\\
    \overline{5} & \overline{6} & \overline{1} & 2 & \overline{3} & \colorbox{white!20}4 & \overline{4} & 3 & \overline{2} & 1 & 6 & 5 \\
    \overline{5} & \overline{1} & \overline{6} & \overline{3} & 2 & \colorbox{white!20}4 & \overline{4} & \overline{2} & 3 & 6 & 1 & 5\\
    \overline{1} & \overline{5} & \overline{3} & \overline{6} & \colorbox{green!35}{$\overline{2}$} & \colorbox{green!35}{$\overline{4}$} & \colorbox{green!35}{$4$} & \colorbox{green!35}{$2$} & 6 & 3 & 5 & 1\\
    \overline{1} & \overline{3} & \overline{5} & \overline{2} & \overline{6} & \overline{4} & \colorbox{white!20}4 & 6 & 2 & 5 & 3 & 1\\
    \overline{1} & \overline{3} & \overline{2} & \overline{5} & \overline{4} & \overline{6} & \colorbox{white!20}6 & 4 & 5 & 2 & 3 & 1\\
    \overline{1} & \overline{2} & \overline{3} & \overline{4} & \overline{5} & \overline{6} & \colorbox{white!20}6 & 5 & 4 & 3 & 2 & 1
    \end{pmatrix}\hspace{.15in}
    \begin{pmatrix}
       1 & 2 & 3 & 4 & 5 & \overline{5} & \overline{4} & \overline{3} & \overline{2} & \overline{1}\\
       1 & 3 & 2 & \colorbox{green!35}{$\overline{4}$} & \colorbox{green!35}{$\overline{5}$} & \colorbox{green!35}{$5$} & \colorbox{green!35}{$4$} & \overline{2} & \overline{3} & \overline{1}\\
       1 & 3 & \overline{4} & \colorbox{white!20}2 & \overline{5} & 5 & \overline{2} & 4 & \overline{3} & \overline{1}\\
       1 & \overline{4} & 3 & \overline{5} & \colorbox{white!20}2 & \overline{2} & 5 & \overline{3} & 4 & \overline{1}\\
       \overline{4} & 1 & \overline{5} & 3 & \colorbox{green!35}{$\overline{2}$} & \colorbox{green!35}{$2$} & \overline{3} & 5 & \overline{1} & 4 \\
       \overline{4} & \overline{5} & 1 & \overline{2} & \colorbox{white!20}3 & \overline{3} & 2 & \overline{1} & 5 & 4 \\
       \overline{4} & \overline{5} & \overline{2} & 1 & \colorbox{white!20}3 & \overline{3} & \overline{1} & 2 & 5 & 4 \\
       \overline{4} & \overline{2} & \overline{5} & \colorbox{green!35}{$\overline{1}$} & \colorbox{green!35}{$\overline{3}$} & \colorbox{green!35}{$3$} & \colorbox{green!35}{$1$} & 5 & 2 & 4\\
       \overline{2} & \overline{4} & \overline{1} & \overline{5} & \overline{3} & \colorbox{white!20}3 & 5 & 1 & 4 & 2\\
       \overline{2} & \overline{1} & \overline{4} & \overline{3} & \overline{5} & \colorbox{white!20}5 & 3 & 4 & 1 & 2 \\
       \overline{1} & \overline{2} & \overline{3} & \overline{4} & \overline{5} & \colorbox{white!20}5 & 4 & 3 & 2 & 1
    \end{pmatrix}
\end{equation*}
\end{footnotesize}
\caption{The sequence $\Pi=\text{DC}^{cs}_{ncgen}(2,2,2)$  and the induced sequence $\Pi'=\text{DC}^{cs}_{ncgen}(2,1,2)$.}
    \label{fig:222_212}
\end{figure}
Furthermore, understanding which induced subsequences of an even-near-critical centrally symmetric allowable sequence are also even-near-critical seems to be key to the full understanding of direction-(near)-critical configurations. For example, Proposition \ref{prop:extreme_expcross}(\ref{prop_part:expcross}) shows that a subsequence  induced by two full crossing substrings remains even-near-critical. In contrast, even when the sequence $\Pi=\text{DC}^{\,cs}_{ncgen}(2,2,2)$ is geometrically realizable by $(\text{Z}5,12,6)$, the induced sequence $\Pi'=\text{DC}^{\,cs}_{ncgen}(2,1,2)$ is not by Theorem \ref{th:(2,1,2)}. Figure \ref{fig:222_212} shows these two allowable sequences and Figure \ref{fig:222_subseq} shows the subsequences of $\Pi$ obtained by removing one point and its conjugate. More precisely, if $S=[6]\cup\overline{[6]}$, then the allowable sequences $\Pi_1:=\Pi|_{S\setminus\{1,\overline{1}\}}$ and $\Pi_3:=\Pi|_{S\setminus\{3,\overline{3}\}}$ are not even-near-critical as their halfperiods have length 12. Comparing them to the sequence $\Pi'$, which actually has a halfperiod of length 10, we can see that they are quite similar. In fact, $\Pi_1$ and $\Pi_3$ are both \emph{semispace equivalent} to $\Pi'$ as defined in \cite{GP84}. While $\Pi_1$ and $\Pi_3$ are geometrically realizable, $\Pi'$ requires a couple of extra pairs of parallel lines making a geometric realization impossible. 

\begin{figure}[h]
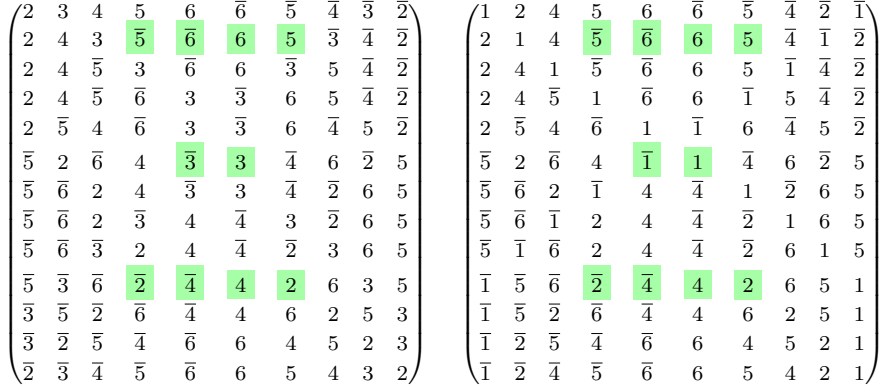

    \centering
\begin{footnotesize}
\begin{equation*}
    \setlength\arraycolsep{4.4pt}
    \renewcommand{\arraystretch}{0.2}
    \begin{pmatrix}
    2 & 3 & 4 & 5 & 6 & \overline{6} & \overline{5} & \overline{4} & \overline{3} & \overline{2} \\
    2 & 4 & 3 & \colorbox{green!35}{$\overline{5}$} & \colorbox{green!35}{$\overline{6}$} & \colorbox{green!35}{$6$} & \colorbox{green!35}{$5$} & \overline{3} & \overline{4} &  \overline{2}\\
    2 & 4  & \overline{5} & \colorbox{white!20}3 & \overline{6} & 6 & \overline{3} & 5 &  \overline{4} & \overline{2}\\
    2 & 4 & \overline{5}  & \overline{6} & \colorbox{white!20}3 & \overline{3} & 6 &  5 & \overline{4} & \overline{2}\\
    2 & \overline{5} & 4 & \overline{6}  & \colorbox{white!20}3 & \overline{3} &  6 & \overline{4} & 5 & \overline{2}\\
    \overline{5} & 2 & \overline{6} & 4 &  \colorbox{green!35}{$\overline{3}$} & \colorbox{green!35}{$3$}  & \overline{4} & 6 & \overline{2} & 5 \\
    \overline{5} & \overline{6} & 2 &  4 & \overline{3} & \colorbox{white!20}3 & \overline{4}  & \overline{2} & 6 & 5\\
    \overline{5} & \overline{6} &  2 & \overline{3} & \colorbox{white!20}4 & \overline{4} & 3 & \overline{2}  & 6 & 5 \\
    \overline{5} &  \overline{6} & \overline{3} & 2 & \colorbox{white!20}4 & \overline{4} & \overline{2} & 3 & 6  & 5\\
     \overline{5} & \overline{3} & \overline{6} & \colorbox{green!35}{$\overline{2}$} & \colorbox{green!35}{$\overline{4}$} & \colorbox{green!35}{$4$} & \colorbox{green!35}{$2$} & 6 & 3 & 5 \\
     \overline{3} & \overline{5} & \overline{2} & \overline{6} & \overline{4} & \colorbox{white!20}4 & 6 & 2 & 5 & 3\\
     \overline{3} & \overline{2} & \overline{5} & \overline{4} & \overline{6} & \colorbox{white!20}6 & 4 & 5 & 2 & 3 \\
     \overline{2} & \overline{3} & \overline{4} & \overline{5} & \overline{6} & \colorbox{white!20}6 & 5 & 4 & 3 & 2 
    \end{pmatrix}
    \hspace{.15in}
    \begin{pmatrix}
    1 & 2  & 4 & 5 & 6 & \overline{6} & \overline{5} & \overline{4} & \overline{2} & \overline{1}\\
    2 & 1 & 4  & \colorbox{green!35}{$\overline{5}$} & \colorbox{green!35}{$\overline{6}$} & \colorbox{green!35}{$6$} & \colorbox{green!35}{$5$}  & \overline{4} & \overline{1} & \overline{2}\\
    2 & 4 & 1 & \overline{5}  & \overline{6} & \colorbox{white!20}6  & 5 & \overline{1} & \overline{4} & \overline{2}\\
    2 & 4 & \overline{5} & 1 & \overline{6}   & \colorbox{white!20}6 & \overline{1} & 5 & \overline{4} & \overline{2}\\
    2 & \overline{5} & 4 & \overline{6} & \colorbox{white!20}1  & \overline{1} & 6 & \overline{4} & 5 & \overline{2}\\
    \overline{5} & 2 & \overline{6} & 4 & \colorbox{green!35}{$\overline{1}$}  & \colorbox{green!35}{$1$} & \overline{4} & 6 & \overline{2} & 5 \\
    \overline{5} & \overline{6} & 2 & \overline{1} & \colorbox{white!20}4   & \overline{4} & 1 & \overline{2} & 6 & 5\\
    \overline{5} & \overline{6} & \overline{1} & 2  & \colorbox{white!20}4 & \overline{4} & \overline{2} & 1 & 6 & 5 \\
    \overline{5} & \overline{1} & \overline{6}  & 2 & \colorbox{white!20}4 & \overline{4} & \overline{2}  & 6 & 1 & 5\\
    \overline{1} & \overline{5}  & \overline{6} & \colorbox{green!35}{$\overline{2}$} & \colorbox{green!35}{$\overline{4}$} & \colorbox{green!35}{$4$} & \colorbox{green!35}{$2$} & 6  & 5 & 1\\
    \overline{1}  & \overline{5} & \overline{2} & \overline{6} & \overline{4} & \colorbox{white!20}4 & 6 & 2 & 5  & 1\\
    \overline{1}  & \overline{2} & \overline{5} & \overline{4} & \overline{6} & \colorbox{white!20}6 & 4 & 5 & 2  & 1\\
    \overline{1} & \overline{2}  & \overline{4} & \overline{5} & \overline{6} & \colorbox{white!20}6 & 5 & 4  & 2 & 1
    \end{pmatrix}
\end{equation*}
\end{footnotesize}
    \caption{The subsequences $\Pi_1$ and $\Pi_3$ of $\text{DC}^{\,cs}_{ncgen}(2,2,2)$ induced by removing the points $\{1,\overline{1}\}$ or $\{3,\overline{3}\}$, respectively.}
    \label{fig:222_subseq}
\end{figure}

\vskip2cm

\newpage
\section*{Appendix}
\begin{enumerate}
    \item The sequence $(1,1,1,1,1,1,1,1)$ in Figure \ref{fig:geomrealizations}(a) as an example for the type $(1,1,\dots,1)$ which corresponds to a regular polygon.

\begin{footnotesize}
\begin{equation*}
    \begin{pmatrix}
    1 & 2 & 3 & 4 & 5 & 6 & 7 & 8 & \overline{8} & \overline{7} & \overline{6} & \overline{5} & \overline{4} & \overline{3} & \overline{2} & \overline{1}\\
    1 & 3 & 2 & 5 & 4 & 7 & 6 & \colorbox{green!35}{$\overline{8}$} & \colorbox{green!35}{$8$} & \overline{6} & \overline{7} & \overline{4} & \overline{5} & \overline{2} & \overline{3} & \overline{1}\\
    3 & 1 & 5 & 2 & 7 & 4 & \colorbox{blue!35}{$\overline{8}$} & \colorbox{blue!35}6 & \colorbox{blue!20}{$\overline{6}$} & \colorbox{blue!20}8 & \overline{4} & \overline{7} & \overline{2} & \overline{5} & \overline{1} & \overline{3}\\
    3 & 5 & 1 & 7 & 2 & \colorbox{blue!35}{$\overline{8}$} & \colorbox{blue!35}4 & \colorbox{green!35}{$\overline{6}$} & \colorbox{green!35}{$6$} & \colorbox{blue!20}{$\overline{4}$} & \colorbox{blue!20}8 & \overline{2} & \overline{7} & \overline{1} & \overline{5} & \overline{3}\\
    5 & 3 & 7 & 1 & \colorbox{blue!35}{$\overline{8}$} & \colorbox{blue!35}2 & \overline{6} & 4 & \overline{4} & 6 & \colorbox{blue!20}{$\overline{2}$} & \colorbox{blue!20}8 & \overline{1} & \overline{7} & \overline{3} & \overline{5}\\
    5 & 7 & 3 & \colorbox{blue!35}{$\overline{8}$} & \colorbox{blue!35}1 & \overline{6} & 2 & \colorbox{green!35}{$\overline{4}$} & \colorbox{green!35}{$4$} & \overline{2} & 6 & \colorbox{blue!20}{$\overline{1}$} & \colorbox{blue!20}8 & \overline{3} & \overline{7} & \overline{5}\\
    7 & 5 & \colorbox{blue!35}{$\overline{8}$} & \colorbox{blue!35}3 & \overline{6} & 1 & \overline{4} & 2 & \overline{2} & 4 & \overline{1} & 6 & \colorbox{blue!20}{$\overline{3}$} & \colorbox{blue!20}8 & \overline{5} & \overline{7}\\
    7 & \colorbox{blue!35}{$\overline{8}$} & \colorbox{blue!35}5 & \overline{6} & 3 & \overline{4} & 1 & \colorbox{green!35}{$\overline{2}$} & \colorbox{green!35}{$2$} & \overline{1} & 4 & \overline{3} & 6 & \colorbox{blue!20}{$\overline{5}$} & \colorbox{blue!20}8 & \overline{7} \\
   \colorbox{blue!35}{$\overline{8}$} & \colorbox{blue!35}7 & \overline{6} & 5 & \overline{4} & 3 & \overline{2} & 1 & \overline{1} & 2 & \overline{3} & 4 & \overline{5} & 6 & \colorbox{blue!20}{$\overline{7}$} & \colorbox{blue!20}8 \\
    \colorbox{blue!35}{$\overline{8}$} & \overline{6} & 7 & \overline{4} & 5 & \overline{2} & 3 & \colorbox{green!35}{$\overline{1}$} & \colorbox{green!35}{$1$} & \overline{3} & 2 & \overline{5} & 4 & \overline{7} & 6 & \colorbox{blue!20}8\\
    \colorbox{blue!35}{$\overline{6}$} & \colorbox{blue!35}{$\overline{8}$} & \overline{4} & 7 & \overline{2} & 5 & \overline{1} & 3 & \overline{3} & 1 & \overline{5} & 2 & \overline{7} & 4 & \colorbox{blue!20}8 & \colorbox{blue!20}6 \\
    \overline{6} & \colorbox{blue!35}{$\overline{4}$} & \colorbox{blue!35}{$\overline{8}$} & \overline{2} & 7 & \overline{1} & 5 & \colorbox{green!35}{$\overline{3}$} & \colorbox{green!35}{$3$} & \overline{5} & 1 & \overline{7} & 2 & \colorbox{blue!20}8 & \colorbox{blue!20}4 & 6\\
    \overline{4} & \overline{6} & \colorbox{blue!35}{$\overline{2}$} & \colorbox{blue!35}{$\overline{8}$} & \overline{1} & 7 & \overline{3} & 5 & \overline{5} & 3 & \overline{7} & 1 & \colorbox{blue!20}8 & \colorbox{blue!20}2 & 6 & 4\\
    \overline{4} & \overline{2} & \overline{6} & \colorbox{blue!35}{$\overline{1}$} & \colorbox{blue!35}{$\overline{8}$} & \overline{3} & 7 & \colorbox{green!35}{$\overline{5}$} & \colorbox{green!35}{$5$} & \overline{7} & 3 & \colorbox{blue!20}8 & \colorbox{blue!20}1 & 6 & 2 & 4\\
    \overline{2} & \overline{4} & \overline{1} & \overline{6} & \colorbox{blue!35}{$\overline{3}$} & \colorbox{blue!35}{$\overline{8}$} & \overline{5} & 7 & \overline{7} & 5 & \colorbox{blue!20}8 & \colorbox{blue!20}3 & 6 & 1 & 4 & 2 \\
    \overline{2} & \overline{1} & \overline{4} & \overline{3} & \overline{6} & \colorbox{blue!35}{$\overline{5}$} & \colorbox{blue!35}{$\overline{8}$} & \colorbox{green!35}{$\overline{7}$} & \colorbox{green!35}{$7$} & \colorbox{blue!20}8 & \colorbox{blue!20}5 & 6 & 3 & 4 & 1 & 2\\
    \overline{1} & \overline{2} & \overline{3} & \overline{4} & \overline{5} & \overline{6} & \colorbox{blue!35}{$\overline{7}$} & \colorbox{blue!35}{$\overline{8}$} & \colorbox{blue!20}8 & \colorbox{blue!20}7 & 6 & 5 & 4 & 3 & 2 & 1
    \end{pmatrix}
\end{equation*}
\end{footnotesize}

\vskip0.5cm
\item The sequence $d=(4,4)$ in Figure \ref{fig:geomrealizations}(b) as an example for the type $(d_1,d_2)$ with $d_1,d_2\geq 2$, which corresponds to an exponential cross.

\begin{footnotesize}
\begin{equation*}
    \begin{pmatrix}
   1 & 2 & 3 & 4 & 5 & 6 & 7 & 8 & \overline{8} & \overline{7} & \overline{6} & \overline{5} & \overline{4} & \overline{3} & \overline{2} & \overline{1}\\
    1 & 2 & 3 & 4 & \colorbox{green!35}{$\overline{5}$} & \colorbox{green!35}{$\overline{6}$} & \colorbox{green!35}{$\overline{7}$} & \colorbox{green!35}{$\overline{8}$} & \colorbox{green!35}8 & \colorbox{green!35}7 & \colorbox{green!35}6 & \colorbox{green!35}5 & \overline{4} & \overline{3} & \overline{2} & \overline{1}\\
    1 & 2 & 3 & \overline{5} & 4 & \overline{6} & \overline{7} & \colorbox{blue!35}{$\overline{8}$} & \colorbox{blue!20}8 & 7 & 6 & \overline{4} & 5 & \overline{3} & \overline{2} & \overline{1} \\
    1 & 2 & \overline{5} & 3 & \overline{6} & 4 & \overline{7} & \colorbox{blue!35}{$\overline{8}$} & \colorbox{blue!20}8 & 7 & \overline{4} & 6 & \overline{3} & 5 & \overline{2} & \overline{1}\\
    1 & \overline{5} & 2 & \overline{6} & 3 & \overline{7} & 4 & \colorbox{blue!35}{$\overline{8}$} & \colorbox{blue!20}8 & \overline{4} & 7 & \overline{3} & 6 & \overline{2} & 5 & \overline{1}\\
    \overline{5} & 1 & \overline{6} & 2 & \overline{7} & 3 & \colorbox{blue!35}{$\overline{8}$} & \colorbox{blue!35}4 & \colorbox{blue!20}{$\overline{4}$} & \colorbox{blue!20}8 & \overline{3} & 7 & \overline{2} & 6 & \overline{1} & 5 \\
    \overline{5} & \overline{6} & 1 & \overline{7} & 2 & \colorbox{blue!35}{$\overline{8}$} & \colorbox{blue!35}3 & 4 & \overline{4} & \colorbox{blue!20}{$\overline{3}$} & \colorbox{blue!20}8 & \overline{2} & 7 & \overline{1} & 6 & 5 \\
    \overline{5} & \overline{6} & \overline{7} & 1 & \colorbox{blue!35}{$\overline{8}$} & \colorbox{blue!35}2 & 3 & 4 & \overline{4} & \overline{3} & \colorbox{blue!20}{$\overline{2}$} & \colorbox{blue!20}8 & \overline{1} & 7 & 6 & 5 \\
    \overline{5} & \overline{6} & \overline{7} & \colorbox{blue!35}{$\overline{8}$} & \colorbox{blue!35}1 & 2 & 3 & 4 & \overline{4} & \overline{3} & \overline{2} & \colorbox{blue!20}{$\overline{1}$} & \colorbox{blue!20}8 & 7 & 6 & 5\\
    \colorbox{white!20}{$\overline{5}$} & \colorbox{white!20}{$\overline{6}$} & \colorbox{white!20}{$\overline{7}$} & \colorbox{blue!35}{$\overline{8}$} & \colorbox{green!35}{$\overline{1}$} & \colorbox{green!35}{$\overline{2}$} & \colorbox{green!35}{$\overline{3}$} & \colorbox{green!35}{$\overline{4}$} & \colorbox{green!35}4 & \colorbox{green!35}3 & \colorbox{green!35}2 & \colorbox{green!35}1 & \colorbox{blue!20}8 & \colorbox{white!20}7 & \colorbox{white!20}6 & \colorbox{white!20}5 \\
     \overline{5} & \overline{6} & \overline{7} & \colorbox{blue!35}{$\overline{1}$} & \colorbox{blue!35}{$\overline{8}$} & \overline{2} & \overline{3} & \overline{4} & 4 & 3 & 2 & \colorbox{blue!20}8 & \colorbox{blue!20}1 & 7 & 6 & 5\\
    \overline{5} & \overline{6} & \overline{1} & \overline{7} & \colorbox{blue!35}{$\overline{2}$} & \colorbox{blue!35}{$\overline{8}$} & \overline{3} & \overline{4} & 4 & 3 & \colorbox{blue!20}8 & \colorbox{blue!20}2 & 7 & 1 & 6 & 5\\
    \overline{5} & \overline{1} & \overline{6} & \overline{2} & \overline{7} & \colorbox{blue!35}{$\overline{3}$} & \colorbox{blue!35}{$\overline{8}$} & \overline{4} & 4 & \colorbox{blue!20}8 & \colorbox{blue!20}3 & 7 & 2 & 6 & 1 & 5\\
    \overline{1} & \overline{5} & \overline{2} & \overline{6} & \overline{3} & \overline{7} & \colorbox{blue!35}{$\overline{4}$} & \colorbox{blue!35}{$\overline{8}$} & \colorbox{blue!20}8 & \colorbox{blue!20}4 & 7 & 3 & 6 & 2 & 5 & 1\\
    \overline{1} & \overline{2} & \overline{5} & \overline{3} & \overline{6} & \overline{4} & \overline{7} & \colorbox{blue!35}{$\overline{8}$} & \colorbox{blue!20}8 & 7 & 4 & 6 & 3 & 5 & 2 & 1 \\
    \overline{1} & \overline{2} & \overline{3} & \overline{5} & \overline{4} & \overline{6} & \overline{7} & \colorbox{blue!35}{$\overline{8}$} & \colorbox{blue!20}8 & 7 & 6 & 4 & 5 & 3 & 2 & 1\\
    \overline{1} & \overline{2} & \overline{3} & \overline{4} & \overline{5} & \overline{6} & \overline{7} & \colorbox{blue!35}{$\overline{8}$} & \colorbox{blue!20}8 & 7 & 6 & 5 & 4 & 3 & 2 & 1
    \end{pmatrix}
\end{equation*}
\end{footnotesize}
\newpage
    \item The sequence $d=(5,3)$ in Figure \ref{fig:geomrealizations}(c) as an example for the type $(d_1,d_2)$ with $d_1,d_2\geq 2$, which corresponds to an exponential cross.

\begin{footnotesize}
    \begin{equation*}
        \begin{pmatrix}
            1 & 2 & 3 & 4 & 5 & 6 & 7 & 8 & \overline{8} & \overline{7} & \overline{6} & \overline{5} & \overline{4} & \overline{3} & \overline{2} & \overline{1}\\
            \colorbox{white!20}1 & \colorbox{white!20}2 & \colorbox{white!20}3 & \colorbox{green!35}{$\overline{4}$} & \colorbox{green!35}{$\overline{5}$} & \colorbox{green!35}{$\overline{6}$} & \colorbox{green!35}{$\overline{7}$} & \colorbox{green!35}{$\overline{8}$} & \colorbox{green!20}8 & \colorbox{green!35}7 & \colorbox{green!35}6 & \colorbox{green!35}5 & \colorbox{green!35}4 & \colorbox{white!20}{$\overline{3}$} & \colorbox{white!20}{$\overline{2}$} & \colorbox{white!20}{$\overline{1}$}\\
            1 & 2 & \overline{4} & 3 & \overline{5} & \overline{6} & \overline{7} & \colorbox{blue!35}{$\overline{8}$} & \colorbox{blue!20}8 & 7 & 6 & 5 & \overline{3} & 4 & \overline{2} & \overline{1}\\
            1 & \overline{4} & 2 & \overline{5} & 3 & \overline{6} & \overline{7} & \colorbox{blue!35}{$\overline{8}$} & \colorbox{blue!20}8 & 7 & 6 & \overline{3} & 5 & \overline{2} & 4 & \overline{1}\\
            \overline{4} & 1 & \overline{5} & 2 & \overline{6} & 3 & \overline{7} & \colorbox{blue!35}{$\overline{8}$} & \colorbox{blue!20}8 & 7 & \overline{3} & 6 & \overline{2} & 5 & \overline{1} & 4\\
            \overline{4} & \overline{5} & 1 & \overline{6} & 2 & \overline{7} & 3 & \colorbox{blue!35}{$\overline{8}$} & \colorbox{blue!20}8 & \overline{3} & 7 & \overline{2} & 6 & \overline{1} & 5 & 4\\
            \overline{4} & \overline{5} & \overline{6} & 1 & \overline{7} & 2 & \colorbox{blue!35}{$\overline{8}$} & \colorbox{blue!35}3 & \colorbox{blue!20}{$\overline{3}$} & \colorbox{blue!20}8 & \overline{2} & 7 & \overline{1} & 6 & 5 & 4\\
            \overline{4} & \overline{5} & \overline{6} & \overline{7} & 1 & \colorbox{blue!35}{$\overline{8}$} & \colorbox{blue!35}2 & 3 & \overline{3} & \colorbox{blue!20}{$\overline{2}$} & \colorbox{blue!20}8 & \overline{1} & 7 & 6 & 5 & 4\\
            \overline{4} & \overline{5} & \overline{6} & \overline{7} & \colorbox{blue!35}{$\overline{8}$} & \colorbox{blue!35}1 & 2 & 3 & \overline{3} & \overline{2} & \colorbox{blue!20}{$\overline{1}$} & \colorbox{blue!20}8 & 7 & 6 & 5 & 4\\
            \overline{4} & \overline{5} & \overline{6} & \overline{7} & \colorbox{blue!35}{$\overline{8}$} & \colorbox{green!35}{$\overline{1}$} & \colorbox{green!35}{$\overline{2}$} & \colorbox{green!35}{$\overline{3}$} & \colorbox{green!35}3 & \colorbox{green!35}2 & \colorbox{green!35}1 & \colorbox{blue!20}8 & 7 & 6 & 5 & 4\\
            \overline{4} & \overline{5} & \overline{6} & \overline{7} & \overline{1} & \colorbox{blue!35}{$\overline{8}$} & \overline{2} & \overline{3} & 3 & 2 & \colorbox{blue!20}8 & 1 & 7 & 6 & 5 & 4\\
            \overline{4} & \overline{5} & \overline{6} & \overline{1} & \overline{7} & \overline{2} & \colorbox{blue!35}{$\overline{8}$} & \overline{3} & 3 & \colorbox{blue!20}8 & 2 & 7 & 1 & 6 & 5 & 4\\
            \overline{4} & \overline{5} & \overline{1} & \overline{6} & \overline{2} & \overline{7} & \overline{3} & \colorbox{blue!35}{$\overline{8}$} & \colorbox{blue!20}8 & 3 & 7 & 2 & 6 & 1 & 5 & 4\\
            \overline{4} & \overline{1} & \overline{5} & \overline{2} & \overline{6} & \overline{3} & \overline{7} & \colorbox{blue!35}{$\overline{8}$} & \colorbox{blue!20}8 & 7 & 3 & 6 & 2 & 5 & 1 & 4\\
            \overline{1} & \overline{4} & \overline{2} & \overline{5} & \overline{3} & \overline{6} & \overline{7} & \colorbox{blue!35}{$\overline{8}$} & \colorbox{blue!20}8 & 7 & 6 & 3 & 5 & 2 & 4 & 1\\
            \overline{1} & \overline{2} & \overline{4} & \overline{3} & \overline{5} & \overline{6} & \overline{7} & \colorbox{blue!35}{$\overline{8}$} & \colorbox{blue!20}8 & 7 & 6 & 5 & 3 & 4 & 2 & 1 \\
            \overline{1} & \overline{2} & \overline{3} & \overline{4} & \overline{5} & \overline{6} & \overline{7} & \colorbox{blue!35}{$\overline{8}$} & \colorbox{blue!20}8 & 7 & 6 & 5 & 4 & 3 & 2 & 1
        \end{pmatrix}
    \end{equation*}
    \end{footnotesize}
    \vspace{.2in}
    \item The sequence $d=(6,2)$ in Figure \ref{fig:geomrealizations}(d) as an example for the type $(d_1,d_2)$ with $d_1,d_2\geq 2$, which corresponds to an exponential cross.
\begin{footnotesize}
    \begin{equation*}
        \begin{pmatrix}
            1 & 2 & 3 & 4 & 5 & 6 & 7 & 8 & \overline{8} & \overline{7} & \overline{6} & \overline{5} & \overline{4} & \overline{3} & \overline{2} & \overline{1}\\
            \colorbox{white!20}1 & \colorbox{white!20}2 & \colorbox{green!35}{$\overline{3}$} & \colorbox{green!35}{$\overline{4}$} & \colorbox{green!35}{$\overline{5}$} & \colorbox{green!35}{$\overline{6}$} & \colorbox{green!35}{$\overline{7}$} & \colorbox{green!35}{$\overline{8}$} & \colorbox{green!35}{$8$} & \colorbox{green!35}{$7$} & \colorbox{green!35}{$6$} & \colorbox{green!35}{$5$} & \colorbox{green!35}{$4$} & \colorbox{green!35}{$3$} & \colorbox{white!20}{$\overline{2}$} & \colorbox{white!20}{$\overline{1}$}\\
            1 & \overline{3} & 2 & \overline{4} & \overline{5} & \overline{6} & \overline{7} & \colorbox{blue!35}{$\overline{8}$} & \colorbox{blue!20}{$8$} & 7 & 6 & 5 & 4 & \overline{2} & 3 & \overline{1}\\
            \overline{3} & 1 & \overline{4} & 2 & \overline{5} & \overline{6} & \overline{7} & \colorbox{blue!35}{$\overline{8}$} & \colorbox{blue!20}{$8$} & 7 & 6 & 5 & \overline{2} & 4 & \overline{1} & 3\\
            \overline{3} & \overline{4} & 1 & \overline{5} & 2 & \overline{6} & \overline{7} & \colorbox{blue!35}{$\overline{8}$} & \colorbox{blue!20}{$8$} & 7 & 6 & \overline{2} & 5 & \overline{1} & 4 & 3 \\
            \overline{3} & \overline{4} & \overline{5} & 1 & \overline{6} & 2 & \overline{7} & \colorbox{blue!35}{$\overline{8}$} & \colorbox{blue!20}{$8$} & 7 & \overline{2} & 6 & \overline{1} & 5 & 4 & 3\\
            \overline{3} & \overline{4} & \overline{5} & \overline{6} & 1 & \overline{7} & 2 & \colorbox{blue!35}{$\overline{8}$} & \colorbox{blue!20}{$8$} & \overline{2} & 7 & \overline{1} & 6 & 5 & 4 & 3\\
            \overline{3} & \overline{4} & \overline{5} & \overline{6} & \overline{7} & 1 & \colorbox{blue!35}{$\overline{8}$} & \colorbox{blue!35}{$2$} & \colorbox{blue!20}{$\overline{2}$} & \colorbox{blue!20}{$8$} & \overline{1} & 7 & 6 & 5 & 4 & 3\\
            \overline{3} & \overline{4} & \overline{5} & \overline{6} & \overline{7} & \colorbox{blue!35}{$\overline{8}$} & \colorbox{blue!35}{$1$} & 2 & \overline{2} & \colorbox{blue!20}{$\overline{1}$} & \colorbox{blue!20}{$8$} & 7 & 6 & 5 & 4 & 3\\
            \overline{3} & \overline{4} & \overline{5} & \overline{6} & \overline{7} & \colorbox{blue!35}{$\overline{8}$} & \colorbox{green!35}{$\overline{1}$} & \colorbox{green!35}{$\overline{2}$} & \colorbox{green!35}{$2$} & \colorbox{green!35}{$1$} & \colorbox{blue!20}{$8$} & 7 & 6 & 5 & 4 & 3\\
            \overline{3} & \overline{4} & \overline{5} & \overline{6} & \overline{7} & \colorbox{blue!35}{$\overline{1}$} & \colorbox{blue!35}{$\overline{8}$} & \overline{2} & 2 & \colorbox{blue!20}{$8$} & \colorbox{blue!20}{$1$} & 7 & 6 & 5 & 4 & 3\\
            \overline{3} & \overline{4} & \overline{5} & \overline{6} & \overline{1} & \overline{7} & \colorbox{blue!35}{$\overline{2}$} & \colorbox{blue!35}{$\overline{8}$} & \colorbox{blue!20}8 & \colorbox{blue!20}{$2$} & 7 & 1 & 6 & 5 & 4 & 3 \\
            \overline{3} & \overline{4} & \overline{5} & \overline{1} & \overline{6} & \overline{2} & \overline{7} & \colorbox{blue!35}{$\overline{8}$} & \colorbox{blue!20}8 & 7 & 2 & 6 & 1 & 5 & 4 & 3 \\
            \overline{3} & \overline{4} & \overline{1} & \overline{5} & \overline{2} & \overline{6} & \overline{7} & \colorbox{blue!35}{$\overline{8}$} & \colorbox{blue!20}8 & 7 & 6 & 2 & 5 & 1 & 4 & 3 \\
            \overline{3} & \overline{1} & \overline{4} & \overline{2} & \overline{5} & \overline{6} & \overline{7} & \colorbox{blue!35}{$\overline{8}$} & \colorbox{blue!20}8 & 7 & 6 & 5 & 2 & 4 & 1 & 3\\
            \overline{1} & \overline{3} & \overline{2} & \overline{4} & \overline{5} & \overline{6} & \overline{7} & \colorbox{blue!35}{$\overline{8}$} & \colorbox{blue!20}8 & 7 & 6 & 5 & 4& 2 & 3 & 1\\
            \overline{1} & \overline{2} & \overline{3} & \overline{4} & \overline{5} & \overline{6} & \overline{7} & \colorbox{blue!35}{$\overline{8}$} & \colorbox{blue!20}8 & 7 & 6 & 5 & 4 & 3 & 2 & 1
        \end{pmatrix}
    \end{equation*}
\end{footnotesize}
\newpage
\item The sequence $(7,1)$ in Figure \ref{fig:geomrealizations}(e) as an example for the type $(d_1,1)$ with $d_1\geq 2$, which corresponds to a bipencil:

\begin{footnotesize}
\begin{equation*}
    \begin{pmatrix}
    1 & 2 & 3 & 4 & 5 & 6 & 7 & 8 & \overline{8} & \overline{7} & \overline{6} & \overline{5} & \overline{4} & \overline{3} & \overline{2} & \overline{1}\\
    \colorbox{white!20}1 & \colorbox{green!35}{$\overline{2}$} & \colorbox{green!35}{$\overline{3}$} & \colorbox{green!35}{$\overline{4}$} & \colorbox{green!35}{$\overline{5}$} & \colorbox{green!35}{$\overline{6}$} & \colorbox{green!35}{$\overline{7}$} & \colorbox{green!35}{$\overline{8}$} & \colorbox{green!35}8 & \colorbox{green!35}7 & \colorbox{green!35}6 & \colorbox{green!35}5 & \colorbox{green!35}4 & \colorbox{green!35}3 & \colorbox{green!35}2 & \colorbox{white!20}{$\overline{1}$}\\
    \overline{2} & 1 & \overline{3} & \overline{4} & \overline{5} & \overline{6} & \overline{7} & \colorbox{blue!35}{$\overline{8}$} & \colorbox{blue!20}8 & 7 & 6 & 5 & 4 & 3 & \overline{1} & 2\\
    \overline{2} & \overline{3} & 1 & \overline{4} & \overline{5} & \overline{6} & \overline{7} & \colorbox{blue!35}{$\overline{8}$} & \colorbox{blue!20}8 & 7 & 6 & 5 & 4 & \overline{1} & 3 & 2\\
    \overline{2} & \overline{3} & \overline{4} & 1 & \overline{5} & \overline{6} & \overline{7} & \colorbox{blue!35}{$\overline{8}$} & \colorbox{blue!20}8 & 7 & 6 & 5 & \overline{1} & 4 & 3 & 2\\
    \overline{2} & \overline{3} & \overline{4} & \overline{5} & 1 & \overline{6} & \overline{7} & \colorbox{blue!35}{$\overline{8}$} & \colorbox{blue!20}8 & 7 & 6 & \overline{1} & 5 & 4 & 3 & 2 \\
    \overline{2} & \overline{3} & \overline{4} & \overline{5} & \overline{6} & 1 & \overline{7} & \colorbox{blue!35}{$\overline{8}$} & \colorbox{blue!20}8 & 7 & \overline{1} & 6 & 5 & 4 & 3 & 2\\
    \overline{2} & \overline{3} & \overline{4} & \overline{5} & \overline{6} & \overline{7} & 1 & \colorbox{blue!35}{$\overline{8}$} & \colorbox{blue!20}8 & \overline{1} & 7 & 6 & 5 & 4 & 3 & 2\\
    \overline{2} & \overline{3} & \overline{4} & \overline{5} & \overline{6} & \overline{7} & \colorbox{blue!35}{$\overline{8}$} & \colorbox{blue!35}1 & \colorbox{blue!20}{$\overline{1}$} & \colorbox{blue!20}8 & 7 & 6 & 5 & 4 & 3 & 2\\
    \overline{2} & \overline{3} & \overline{4} & \overline{5} & \overline{6} & \overline{7} & \colorbox{blue!35}{$\overline{8}$} & \colorbox{green!35}{$\overline{1}$} & \colorbox{green!35}1 & \colorbox{blue!20}8 & 7 & 6 & 5 & 4 & 3 & 2\\
    \overline{2} & \overline{3} & \overline{4} & \overline{5} & \overline{6} & \overline{7} & \colorbox{blue!35}{$\overline{1}$} & \colorbox{blue!35}{$\overline{8}$} & \colorbox{blue!20}8 & \colorbox{blue!20}{$1$} & 7 & 6 & 5 & 4 & 3 & 2\\
    \overline{2} & \overline{3} & \overline{4} & \overline{5} & \overline{6} & \overline{1} & \overline{7} & \colorbox{blue!35}{$\overline{8}$} & \colorbox{blue!20}8 & 7 & 1 & 6 & 5 & 4 & 3 & 2\\
    \overline{2} & \overline{3} & \overline{4} & \overline{5} & \overline{1} & \overline{6} & \overline{7} & \colorbox{blue!35}{$\overline{8}$} & \colorbox{blue!20}8 & 7 & 6 & 1 & 5 & 4 & 3 & 2\\
    \overline{2} & \overline{3} & \overline{4} & \overline{1} & \overline{5} & \overline{6} & \overline{7} & \colorbox{blue!35}{$\overline{8}$} & \colorbox{blue!20}8 & 7 & 6 & 5 & 1 & 4 & 3 & 2\\
    \overline{2} & \overline{3} & \overline{1} & \overline{4} & \overline{5} & \overline{6} & \overline{7} & \colorbox{blue!35}{$\overline{8}$} & \colorbox{blue!20}8 & 7 & 6 & 5 & 4 & 1 & 3 & 2\\
    \overline{2} & \overline{1} & \overline{3} & \overline{4} & \overline{5} & \overline{6} & \overline{7} & \colorbox{blue!35}{$\overline{8}$} & \colorbox{blue!20}8 & 7 & 6 & 5 & 4 & 3 & 1 & 2\\
    \overline{1} & \overline{2} & \overline{3} & \overline{4} & \overline{5} & \overline{6} & \overline{7} & \colorbox{blue!35}{$\overline{8}$} & \colorbox{blue!20}8 & 7 & 6 & 5 & 4 & 3 & 2 & 1
    \end{pmatrix}
\end{equation*}
\end{footnotesize}
\vskip0.5cm
\item The sequence $(6,1,1)$ in Figure \ref{fig:geomrealizations}(f) as an example for the type $(d_1,1,1)$ with $d_1\geq 2$, which corresponds to a tricolumnar arrangement:
\begin{footnotesize}
\begin{equation*}
    \begin{pmatrix}
   1 & 2 & 3 & 4 & 5 & 6 & 7 & 8 & \overline{8} & \overline{7} & \overline{6} & \overline{5} & \overline{4} & \overline{3} & \overline{2} & \overline{1}\\
   \colorbox{white!20}2 & \colorbox{white!20}1 & \colorbox{green!35}{$\overline{3}$} & \colorbox{green!35}{$\overline{4}$} & \colorbox{green!35}{$\overline{5}$} & \colorbox{green!35}{$\overline{6}$} & \colorbox{green!35}{$\overline{7}$} & \colorbox{green!35}{$\overline{8}$} & \colorbox{green!35}8 & \colorbox{green!35}7 & \colorbox{green!35}6 & \colorbox{green!35}5 & \colorbox{green!35}4 & \colorbox{green!35}3 & \colorbox{white!20}{$\overline{1}$} & \colorbox{white!20}{$\overline{2}$}\\
   2 & \overline{3} & 1 & \overline{4} & \overline{5} & \overline{6} & \overline{7} & \colorbox{blue!35}{$\overline{8}$} & \colorbox{blue!20}8 & 7 & 6 & 5 & 4 & \overline{1} & 3 & \overline{2}\\
   \overline{3} & 2 & \overline{4} & 1 & \overline{5} & \overline{6} & \overline{7} & \colorbox{blue!35}{$\overline{8}$} & \colorbox{blue!20}8 & 7 & 6 & 5 & \overline{1} & 4 & \overline{2} & 3\\
   \overline{3} & \overline{4} & 2 & \overline{5} & 1 & \overline{6} & \overline{7} & \colorbox{blue!35}{$\overline{8}$} & \colorbox{blue!20}8 & 7 & 6 & \overline{1} & 5 & \overline{2} & 4 & 3 \\
   \overline{3} & \overline{4} & \overline{5} & 2 & \overline{6} & 1 & \overline{7} & \colorbox{blue!35}{$\overline{8}$} & \colorbox{blue!20}8 & 7 & \overline{1} & 6 & \overline{2} & 5 & 4 & 3 \\
   \overline{3} & \overline{4} & \overline{5} & \overline{6} & 2 & \overline{7} & 1 & \colorbox{blue!35}{$\overline{8}$} & \colorbox{blue!20}8 & \overline{1} & 7 & \overline{2} & 6 & 5 & 4 & 3\\
   \overline{3} & \overline{4} & \overline{5} & \overline{6} & \overline{7} & 2 & \colorbox{blue!35}{$\overline{8}$} & \colorbox{blue!35}1 & \colorbox{blue!20}{$\overline{1}$} & \colorbox{blue!20}8 & \overline{2} & 7 & 6 & 5 & 4 & 3\\
   \overline{3} & \overline{4} & \overline{5} & \overline{6} & \overline{7} & \colorbox{blue!35}{$\overline{8}$} & \colorbox{blue!35}2 & \colorbox{green!35}{$\overline{1}$} & \colorbox{green!35}1 & \colorbox{blue!20}{$\overline{2}$} & \colorbox{blue!20}8 & 7 & 6 & 5 & 4 & 3 \\
   \overline{3} & \overline{4} & \overline{5} & \overline{6} & \overline{7} & \colorbox{blue!35}{$\overline{8}$} & \overline{1} & 2 & \overline{2} & 1 & \colorbox{blue!20}8 & 7 & 6 & 5 & 4 & 3 \\
   \overline{3} & \overline{4} & \overline{5} & \overline{6} & \overline{7} & \colorbox{blue!35}{$\overline{1}$} & \colorbox{blue!35}{$\overline{8}$} & \colorbox{green!35}{$\overline{2}$} & \colorbox{green!35}2 & \colorbox{blue!20}8 & \colorbox{blue!20}{$1$} & 7 & 6 & 5 & 4 & 3 \\
   \overline{3} & \overline{4} & \overline{5} & \overline{6} & \overline{1} & \overline{7} & \colorbox{blue!35}{$\overline{2}$} & \colorbox{blue!35}{$\overline{8}$} & \colorbox{blue!20}8 & \colorbox{blue!20}{$2$} & 7 & 1 & 6 & 5 & 4 & 3\\
   \overline{3} & \overline{4} & \overline{5} & \overline{1} & \overline{6} & \overline{2} & \overline{7} & \colorbox{blue!35}{$\overline{8}$} & \colorbox{blue!20}8 & 7 & 2 & 6 & 1 & 5 & 4 & 3\\
   \overline{3} & \overline{4} & \overline{1} & \overline{5} & \overline{2} & \overline{6} & \overline{7} & \colorbox{blue!35}{$\overline{8}$} & \colorbox{blue!20}8 & 7 & 6 & 2 & 5 & 1 & 4 & 3\\
   \overline{3} & \overline{1} & \overline{4} & \overline{2} & \overline{5} & \overline{6} & \overline{7} & \colorbox{blue!35}{$\overline{8}$} & \colorbox{blue!20}8 & 7 & 6 & 5 & 2 & 4 & 1 & 3\\
   \overline{1} & \overline{3} & \overline{2} & \overline{4} & \overline{5} & \overline{6} & \overline{7} & \colorbox{blue!35}{$\overline{8}$} & \colorbox{blue!20}8 & 7 & 6 & 5 & 4 & 2 & 3 & 1\\
   \overline{1} & \overline{2} & \overline{3} & \overline{4} & \overline{5} & \overline{6} & \overline{7} & \colorbox{blue!35}{$\overline{8}$} & \colorbox{blue!20}8 & 7 & 6 & 5 & 4 & 3 & 2 & 1
    \end{pmatrix}
\end{equation*}
\end{footnotesize}
\end{enumerate}
\end{document}